\documentclass[11pt]{article}
\usepackage[top=0.75in, bottom=1in, left=3cm, right=3cm]{geometry}
\usepackage[nosort]{cite}

\usepackage{amssymb,amsfonts,amsmath,amsthm}
\allowdisplaybreaks[4]
\usepackage{esint}
\usepackage{braket}
\usepackage{extarrows}
\usepackage{tensor}
\usepackage{bm}
\usepackage[shortcuts]{extdash}
\usepackage{mathrsfs}
\usepackage{dsfont}
\everymath{\displaystyle}
\usepackage{tikz-cd}
\usepackage{thm-restate}

\usepackage[colorlinks=true,linkcolor=black,citecolor=black,CJKbookmarks=true]{hyperref}
\usepackage{enumitem}

\usepackage{thmtools}
\declaretheoremstyle[
bodyfont=\itshape
]{mystyle}
\declaretheorem[style=mystyle,name=Theorem]{theorem}
\declaretheorem[style=mystyle,name=Lemma,sibling=theorem]{lemma}
\declaretheorem[style=mystyle,name=Corollary,sibling=theorem]{corollary}
\declaretheorem[style=mystyle,name=Proposition,sibling=theorem]{proposition}

\declaretheoremstyle[
]{mystyle1}

\declaretheorem[style=mystyle1,name=Definition,sibling=theorem]{definition}
\declaretheorem[style=mystyle1,name=Remark,sibling=theorem]{remark}
\declaretheorem[style=mystyle,name=Example,sibling=theorem]{example}
\usepackage{microtype}

\makeatletter
\renewcommand{\@fnsymbol}[1]{\arabic{footnote}}
\makeatother

\allowdisplaybreaks[4]
\usepackage{cite}		
\usepackage{xcolor}
\definecolor{mred}{RGB}{220,20,60}
\definecolor{green}{RGB}{34,139,34}
\definecolor{blue}{RGB}{0,136,199}
\definecolor{purple}{RGB}{41,18,171}

\begin{document}
	\boldmath
	\title{\centering\textbf{K-energy for the coupled constant scalar curvature K\"{a}hler equation, I: Slope formula and semistability}}
	\author{By \text{Akito Futaki\footnote{Yau Mathematical Sciences Center, Tsinghua University, Beijing 100084, China, futaki@tsinghua.edu.cn} } and \text{Jianwei Shi\footnote{Department of Mathematical Sciences, Tsinghua University, Beijing 100084, China, shijw23@mails.tsinghua.edu.cn}}}
\date{\today}
\maketitle
\unboldmath
\begin{abstract}
		In this paper we study coupled constant scalar curvature K\"ahler metrics 
			defined by Datar-Pingali extending coupled K\"ahler-Einstein metrics defined by Hultgren-Witt Nystr\"om. We consider two K-energy
			functionals of which one (denoted $\mathcal{M}_{\operatorname{ccscK}}$) integrates the moment map considered by Datar-Pingali and the other 
			(denoted $\mathcal{M}'$) is
			designed so that its slope at infinity for subgeodesic rays for ample test configurations with reduced central fiber
			coincides with the Donaldson-Futaki invariant considered by Hultgren-Witt Nystr\"om and Datar-Pingali. We also extend $\mathcal{M}_{\operatorname{ccscK}}$ to $\mathcal{H}^{1,\bar{1}}$ and the finite energy class $\mathcal{E}^1$, and prove its convexity along tuples of weak geodesics and finite energy geodesics, respectively. As a corollary, we obtain a semistability condition necessary for the existence of ccscK metrics.
\end{abstract}
\boldmath
\tableofcontents
\unboldmath
\section{Introduction}
\label{sec:1}
Finding canonical metrics has been one of the main subjects in K\"ahler geometry. The celebrated works on K\"ahler-Einstein metrics are \cite{Yau1977,Yau1978Comm,Aubin1976Paris,Tian2015CPAM,chen2015kahler3} and more recent results on constant scalar curvature K\"ahler (cscK for short) metrics are typically \cite{boucksom2025yautiandonaldsonconjectureweightedcsck,darvas2025ytdcorrespondenceconstantscalar,trusiani2026solutionyautiandonaldsonconjecturespecial}. 

In \cite{hultgren2019coupled}, Hultgren and Witt Nystr\"{o}m defined a new candidate for canonical K\"ahler metrics called the coupled K\"ahler-Einstein (cKE for short) metrics.
There is an example of a Fano 4-fold which has reductive automorphism group but does not admit a K\"ahler-Einstein metric as shown in \cite[Section 3]{Futaki1983Invent}. Hultgren proved in \cite[Corollary 4]{hultgren2019AnP} that the same manifold carries coupled K\"ahler-Einstein metric, showing the possibility that coupled K\"ahler-Einstein metrics could be
another candidate for canonical K\"{a}hler metrics on Fano manifolds, see also \cite{FutakiZhang2021}. More recently, Fujita-Hashimoto \cite{fujita2024coupleddingstabilityyautiandonaldson} obtained a kind of K-stability characterization of the existence of coupled K\"ahler-Einstein metrics.

The cKE metrics are formulated as follows. Let $X$ be a compact complex manifold of complex dimension $n$, we say that a tuple of K\"{a}hler metrics $\omega_i,i=0,...,m$ satisfies the \textit{coupled K\"{a}hler-Einstein equation} if\begin{align}\label{cKE}
	\operatorname{Ric}(\omega_0)=\cdots=\operatorname{Ric}(\omega_m)=\epsilon\sum_i \omega_i, \ \epsilon = 1\text{ or }-1
\end{align}where $\operatorname{Ric}(\omega_i)$ denotes the Ricci form of $\omega_i$. This equation naturally requires that $c_1(X)=\epsilon\sum_i [\omega_i]$. In particular, for the case $\epsilon=-1$, $K_X$ should be ample, and \cite[Theorem 1.3]{hultgren2019coupled} provides a generalization of the Aubin-Yau theorem. 

In the case of $\epsilon=1$, the equation\begin{align}\label{cKE1}
	\operatorname{Ric}(\omega_0)=\cdots=\operatorname{Ric}(\omega_k)=\sum_i \omega_i
\end{align}demands that $X$ should be Fano, and does not always have a solution. Recall in the classical setting, the K\"{a}hler-Einstein equation $\operatorname{Ric}(\omega)=\omega$ is just equivalent to the following fourth-order partial differential equation, known as the constant scalar curvature K\"{a}hler (cscK) equation:\begin{align*}
	S_{\omega}=n,
\end{align*}where $S_\omega$ denotes the scalar curvature of $\omega$. In the coupled setting, Datar and Pingali \cite{datar2020coupled} investigated a generalized system for the cscK equation, formulated as follows:\begin{equation}\label{ccscK}
	\begin{aligned}
		&\frac{\omega_0^n}{V_0}=\cdots=\frac{\omega_m^n}{V_m}\\
		&S_{\omega_0}=\operatorname{tr}_{\omega_0}\left(\sum_{i=0}^m\omega_i\right)+\hat{S},
	\end{aligned}
\end{equation}where $V_i=[\omega_i]^n$ and $\hat{S}$ is a topological constant. Under the assumption that $c_1(X)=\sum_{i=0}^m[\omega_i]$, this equation turns out to be equivalent to (\ref{cKE1}).

The equation (\ref{ccscK}) is called the \textit{coupled constant scalar curvature K\"{a}hler equation} (ccscK for short) and the solutions are called \textit{ccscK metrics}. Building upon the fundamental theories of the ccscK and cKE equations, this paper focuses on a functional $\mathcal{M}_{\operatorname{ccscK}}$ called the \textit{coupled K-energy} or \textit{coupled Mabuchi energy}, which integrates the moment map for the ccscK equation. This functional is suggested in \cite[Remark 12]{datar2020coupled} (see also \cite{Lee2023AJM}). Our first result is a Chen-Tian formula (\ref{our}) of $\mathcal{M}_{\operatorname{ccscK}}$. Let $X$ be an $n$-dimensional complex manifold, with K\"{a}hler classes $\alpha_i,i=0,\cdots,m$, then for any $\omega_i\in \alpha_i$ and any strictly $\omega_i$-plurisubharmonic functions $\varphi_i$, we have
\begin{align*}
	&\mathcal{M}_{\operatorname{ccscK}}((\omega_{i}+dd^c\varphi_i),(\omega_{i}))= \mathcal{M}_{\operatorname{cscK}}(\omega_{0}+dd^c\varphi_0,\omega_{0})+\frac{S}{V_0(n+1)!}\braket{\varphi_0,\cdots,\varphi_0}_{(\omega_0,\cdots,\omega_0)}\\ &\quad -\sum_{i=1}^m\frac{1}{V_i(n+1)!}\braket{\varphi_{i},\cdots,\varphi_{i}}_{(\omega_i,\cdots,\omega_i)}+\frac{1}{V_0n!}\sum_{i=1}^m\braket{\varphi_i,\varphi_0,\cdots,\varphi_0}_{(\omega_i,\omega_0,\cdots,\omega_0)},
\end{align*}where $\braket{\varphi_0,\cdots,\varphi_0}_{(\omega_0,\cdots,\omega_0)}$ denotes the multivariate energy functional (also known as analytic Deligne
pairing) as defined in Definition \ref{delignepairing} and $\mathcal{M}_{\operatorname{cscK}}$ is the well-known K-energy for the cscK
metrics, $V_i=\alpha_i^n$ and $S:=-n\frac{\sum_{i=1}^m\alpha_i\cdot\alpha_0^{n-1}}{\alpha_0^n}$. See \eqref{CT1}, \eqref{CT2}, \eqref{CT3} and \eqref{finiteenergyextensionformula} for more
explicit expressions of the last three terms of $\mathcal{M}_{\operatorname{ccscK}}((\omega_{i}+dd^c\varphi_i),(\omega_{i}))$.

From this formula, we obtain the asymptotic behavior of $\mathcal{M}_{\operatorname{ccscK}}$, by utilizing the results of \cite{boucksom2019} within the analytic framework established by Sj\"{o}str\"{o}m Dyrefelt \cite{sjostrom2018k}. More precisely, we establish the following results, in both the algebraic and cohomological settings.
\begin{theorem}\label{our2}
	Let $X$ be a compact complex manifold with K\"ahler forms $\omega_0,\cdots,\omega_m$ on it, and $(\mathcal{X}_i,\mathcal{A}_i)$ be (compactified) relatively K\"ahler cohomological test conﬁgurations for $(X,[\omega_i])$ respectively. Suppose that all of these test configurations are smooth and dominate $X\times\mathbb{P}^1$. Then for any tuple of subgeodesic rays $(\varphi_i(t))_{t\geq 0}\subset\operatorname{PSH}(X,\omega_i)$, each $\mathcal{C}^{1,\bar{1}}$-compatible (in sense of Definition \ref{compatible}) with the corresponding test configuration $(\mathcal{X}_i,\mathcal{A}_i)$, $i=0,\ldots,m$, we have
	\begin{equation}\label{ourslope}
		\begin{aligned}
			\lim_{t\to\infty}\frac{\mathcal{M}_{\operatorname{ccscK}}\left((\omega_i+dd^c\varphi_i(t)),(\omega_i)\right)}{t}&=\operatorname{M}^{\operatorname{NA}}(\mathcal{X}_0, \mathcal{A}_0)+\frac{S}{V_0(n+1)!}(\mathcal{A}_0^{n+1})\\ &\quad -\sum_{i=1}^m\frac{1}{V_i(n+1)!}(\mathcal{A}_{i}^{n+1})+\frac{1}{V_0n!}\sum_{i=1}^m(\mathcal{A}_i\cdot\mathcal{A}_0^n),
		\end{aligned}
	\end{equation}where $\operatorname{M}^{\operatorname{NA}}(\mathcal{X}_0,\mathcal{A}_0)$ denotes the non-Archimedean Mabuchi functional of $(\mathcal{X}_0,\mathcal{A}_0)$, defined by \begin{align*}
		\operatorname{M}^{\operatorname{NA}}(\mathcal{X}_0,\mathcal{A}_0)=&\ \frac{n(-K_X\cdot [\omega_0]^{n-1})/[\omega_0]^{2n}}{(n+1)!}(\mathcal{A}_0^{n+1})+\frac{1}{V_0n!}\left(K_{\mathcal{X}_0/\mathbb{P}^1}\cdot\mathcal{A}_0^n\right)\\&+\frac{1}{V_0n!}((\mathcal{X}_0(0)_{\operatorname{red}}-\mathcal{X}_0(0))\cdot \mathcal{A}_0^n)
	\end{align*} where $\mathcal{X}_0(0)$ denotes the central fiber of $\mathcal{X}_0$ (c.f. \cite{boucksom2017uniform,boucksom2019} and \cite[Deﬁnition 3.15]{sjostrom2018k}). The intersection numbers on the right hand side of (\ref{ourslope}) are given by pulling back the cohomology classes $\mathcal{A}_i$ to a smooth test configuration $\mathcal{X}$ that dominates all $\mathcal{X}_i$ (c.f. \cite[Section 6.6]{boucksom2017uniform} and \cite[Section 3.1]{sjostrom2018k}).
\end{theorem}
\begin{theorem}\label{our2alg}
	Let $X$ be a compact complex manifold with ample line bundles $L_0,\cdots,L_m$ on it, and  $(\mathcal{X}_i,\mathcal{L}_i)$ be (compactified) ample algebraic test conﬁgurations for $(X,L_i)$ respectively. Suppose that all of these test configurations are smooth and dominate $X\times\mathbb{P}^1$. Then for any smooth strictly positive metrics $\Phi_i$ on $\mathcal{L}_i$ near the central ﬁber, we have
	\begin{equation}\label{ourslopealg}
		\begin{aligned}
			\lim_{t\to\infty}\frac{\mathcal{M}_{\operatorname{ccscK}}\left((\omega_i(t)),(\omega_i)\right)}{t}&=\operatorname{M}^{\operatorname{NA}}(\mathcal{X}_0, \mathcal{L}_0)+\frac{S}{V_0(n+1)!}(\mathcal{L}_0^{n+1})\\ &\quad -\sum_{i=1}^m\frac{1}{V_i(n+1)!}(\mathcal{L}_{i}^{n+1})+\frac{1}{V_0n!}\sum_{i=1}^m(\mathcal{L}_i\cdot\mathcal{L}_0^n),
		\end{aligned}
	\end{equation}where $\omega_i$ is a reference metric in $c_1(L_i)$, $\omega_i(t) = dd^c\phi_i(t)$ and $\phi_i(t)$ is the ray of smooth positive metrics on $L_i$ corresponding to $\Phi_i$. The intersection numbers on the right hand side of (\ref{ourslopealg}) are given by pulling back the cohomology classes $\mathcal{L}_i$ to a smooth test configuration $\mathcal{X}$ that dominates all $\mathcal{X}_i$ (c.f. \cite[Section 6.6]{boucksom2017uniform} and \cite[Section 3.1]{sjostrom2018k}).
\end{theorem}Under the cohomological assumption $c_1(X)=\sum_{i=0}^{m}[\omega_i]$, derived from the cKE equation (\ref{cKE1}), we have $S-n=-n\frac{\sum_{i=0}^m[\omega_i]\cdot[\omega_0]^{n-1}}{[\omega_0]^n}=n\frac{K_X\cdot[\omega_0]^{n-1}}{[\omega_0]^n}$ and thus\begin{align*}
	\lim_{t\to\infty}\frac{\mathcal{M}_{\operatorname{ccscK}}\left((\omega_i(t)),(\omega_i)\right)}{t}&=-\sum_{i=0}^m\frac{1}{V_i(n+1)!}(\mathcal{L}_{i}^{n+1})+\frac{1}{V_0n!}\left(\left(K_{\mathcal{X}_0/\mathbb{P}^1}+\sum_{i=0}^m\mathcal{L}_i\right)\cdot\mathcal{L}_0^n\right),
\end{align*}when $\mathcal{X}_0(0)$ is reduced, i.e.  $((\mathcal{X}_0(0)_{\operatorname{red}}-\mathcal{X}_0(0))\cdot \mathcal{L}_0^n)=0$).
This formula is quite similar to the DF-invariant given by \cite{hultgren2019coupled} (see also \cite[Proposition 19]{datar2020coupled}), but not exactly the same. To explain this difference, 
we construct in section \ref{subsec:6} another K-energy $\mathcal{M}'$ expressed as (\ref{M'}) that matches their DF-invariant in \cite{hultgren2019coupled} and \cite{datar2020coupled}
	in the sense that the slope (i.e. non-Archimedean functional) of $\mathcal{M}'$ is their DF-invariant, see (\ref{slopeforHWNDF})
	in Proposition \ref{our5}. To compare $\mathcal{M}_{\operatorname{ccscK}}$ and $\mathcal{M}'$, it is convenient to use an auxiliary energy functional 
	$\mathcal{N}_{\operatorname{cKE}}$, which is given in Definition \ref{NcKE} (originally introduced in \cite[(14)]{hultgren2019coupled}), together with the entropy $\operatorname{Ent}$ as described in (\ref{Entropy}). We prove that $\mathcal{M}_{\operatorname{ccscK}}$ and $\mathcal{M}'$ are both greater than $\mathcal{N}_{\operatorname{cKE}}$ by one entropy part respectively as described in the next theorem.
\begin{theorem}\label{ourt2}Suppose that $c_1(X)=\sum_{i=0}^m[\omega_i]$ and let $\Omega$ be the K\"{a}hler form in $c_1(X)$ such that $\Omega:\operatorname{Ric}(\Omega)=\sum_{i=0}^m\omega_i$. Then\\
	$(\mathrm i)$ \begin{equation}\label{my3}
		\begin{aligned}
			&\mathcal{M}_{\operatorname{ccscK}}((\omega_{i}+dd^c\varphi_i),(\omega_{i}))\\&=\mathcal{N}_{\operatorname{cKE}}((\omega_{i}+dd^c\varphi_i),(\omega_{i}))+\operatorname{Ent}\left(\frac{(\omega_0+dd^c\varphi_0)^n}{\int_X\omega_0^n},\frac{e^{-\sum_{i=0}^m\varphi_i}\Omega^n}{\int_Xe^{-\sum_{i=0}^m\varphi_i}\Omega^n}\right)-R
		\end{aligned}
	\end{equation}holds for any $\varphi_i\in \mathcal{H}(\omega_i)$, where the relative entropy $\operatorname{Ent}(\mu,\nu)$ is defined by (\ref{Entropy}) and $R:=\operatorname{Ent}\left(\frac{\omega_0^n}{\int_X\omega_0^n},\frac{\Omega^n}{\int_X\Omega^n}\right)$.\\
	$(\mathrm{ii})$ \begin{equation}\label{my3'}
		\begin{aligned}
			&\mathcal{M}'((\omega_{i}+dd^c\varphi_i),(\omega_{i}))\\&=\mathcal{N}_{\operatorname{cKE}}((\omega_{i}+dd^c\varphi_i),(\omega_{i}))+\operatorname{Ent}\left(\frac{\left(\sum_{i=0}^m\omega_i+dd^c\varphi_i\right)^n}{\int_X\left(\sum_{i=0}^m\omega_i\right)^n},\frac{e^{-\sum_{i=0}^m\varphi_i}\Omega^n}{\int_Xe^{-\sum_{i=0}^m\varphi_i}\Omega^n}\right)-R'
		\end{aligned}
	\end{equation}holds for any $\varphi_i\in \mathcal{H}(\omega_i)$, where $R':=\operatorname{Ent}\left(\frac{\left(\sum_{i=0}^m\omega_i\right)^n}{\int_X\left(\sum_{i=0}^m\omega_i\right)^n},\frac{\Omega^n}{\int_X\Omega^n}\right)$.
\end{theorem}
We also show that the functionals $\mathcal{M}_{\operatorname{ccscK}},\mathcal{M}'$ and $\mathcal{N}_{\operatorname{cKE}}$ differ by constants along curves of K\"ahler forms obtained by pullback under holomorphic automorphisms. In particular, along the curve induced by $\exp(t\xi)$, where $\xi$ is a holomorphic vector field, they have the same constant slope, given by the coupled Futaki invariant evaluated at $\xi$.

The critical points of $\mathcal{M}_{\operatorname{ccscK}}$ are solutions of (\ref{ccscK}), i.e. ccscK metrics while 
the critical points of $\mathcal{M}'$ are expressed as solutions of 
\begin{align*}
	\frac{1}{V}(n+1-S_{\sum_{i=0}^m\omega_i,J})\left(\sum_{i=0}^m\omega_i\right)^n=\frac{\omega_i^n}{V_i}.
\end{align*}
Even on Fano manifolds, it is not clear how the cKE equation is related to this equation. The convexity of $\mathcal{M}'$ is also unclear. See Remark \ref{remarkM'} for more detail. However we can show the following.
\begin{corollary}\label{our4}
	Suppose that $c_1(X)=\sum_{i=0}^m[\omega_i]$ and let $\Omega$ be the K\"{a}hler form in $c_1(X)$ such that $\Omega:\operatorname{Ric}(\Omega)=\sum_{i=0}^m\omega_i$. Then\\
	$(\mathrm{i})$ If $([\omega_i])$ admits a coupled Kähler–Einstein metric, then $\mathcal{M}_{\operatorname{ccscK}}(\cdot,(\omega_i))$ and $\mathcal{M}'(\cdot,(\omega_i))$ are bounded from below;\\
	$(\mathrm{ii})$ If $[\omega_i]=L_i$ for ample line bundles $L_i$ and $([\omega_i])$ admits a coupled Kähler–Einstein metric, then for any smooth (compactified) ample algebraic test conﬁgurations $(\mathcal{X}_i,\mathcal{L}_i)$ of $(X,L_i)$ that dominates $X\times\mathbb{P}^1$, we have\begin{align}\label{his1}
		-\sum_{i=0}^m\frac{1}{V_i(n+1)!}(\mathcal{L}_{i}^{n+1})+\frac{1}{V_0n!}\left(\left(K_{\mathcal{X}_0/\mathbb{P}^1}+\sum_{i=0}^m\mathcal{L}_i\right)\cdot\mathcal{L}_0^n\right)\geq 0;
	\end{align}\noindent $(\mathrm{iii})$ If $[\omega_i]=L_i$ for ample line bundles $L_i$ and $([\omega_i])$ admits a coupled Kähler–Einstein metric, then for any smooth (compactified) ample algebraic test conﬁgurations $(\mathcal{X},\mathcal{L}_i)$ of $(X,L_i)$ that dominates $X\times\mathbb{P}^1$ (sharing a common total space $\mathcal{X}$), we have\begin{align}\label{his2}
	-\sum_{i=0}^m\frac{1}{V_i(n+1)!}(\mathcal{L}_{i}^{n+1})+\frac{1}{Vn!}\left(\left(K_{\mathcal{X}/\mathbb{P}^1}+\sum_{i=0}^m\mathcal{L}_i\right)\cdot\left(\sum_{i=0}^m\mathcal{L}_i\right)^n\right)\geq 0,
	\end{align}where $V=(-K_X)^n$.
\end{corollary}
In fact, the conclusion (ii) of Corollary \ref{our4} holds in general ccscK and cohomological case, namely:
\begin{theorem}\label{our6}
	Let $X$ be a compact complex manifold with K\"{a}hler classes $\alpha_i,i=0,\cdots,m$, on it.\\
		$(\mathrm{i})$ If $(\alpha_i)_{i=0}^m$ admits a system of coupled cscK metrics $(\omega_i')_{i=0}^m$, then \begin{align}\label{BoundedBelow}
		\mathcal{M}_{\operatorname{ccscK}}((\omega_i''),(\omega_i))\geq \mathcal{M}_{\operatorname{ccscK}}((\omega_i'),(\omega_i))
	\end{align}for any smooth K\"{a}hler forms $\omega_i''\in \alpha_i$. Moreover, for any smooth (compactified) relatively K\"{a}hler cohomological test conﬁgurations $(\mathcal{X}_i,\mathcal{A}_i)$ of $(X,\alpha_i)$ that dominates $X\times\mathbb{P}^1$, we have\begin{align}\label{cohomologicalcoupledKsemistable}
		\operatorname{M}^{\operatorname{NA}}(\mathcal{X}_0, \mathcal{A}_0)+\frac{S}{V_0(n+1)!}(\mathcal{A}_0^{n+1}) -\sum_{i=1}^m\frac{1}{V_i(n+1)!}(\mathcal{A}_{i}^{n+1})+\frac{1}{V_0n!}\sum_{i=1}^m(\mathcal{A}_i\cdot\mathcal{A}_0^n)\geq 0.
		\end{align}The intersection numbers here are given by pulling back the cohomology classes $\mathcal{A}_i$ to a smooth test configuration $\mathcal{X}$ that dominates all $\mathcal{X}_i$ (c.f. \cite[Section 3.1]{sjostrom2018k}).\\
	$(\mathrm{ii})$ If further $\alpha_i=L_i$ for ample line bundles $L_i$, then for any smooth (compactified) ample algebraic test conﬁgurations $(\mathcal{X}_i,\mathcal{L}_i)$ of $(X,L_i)$ that dominates $X\times\mathbb{P}^1$, we have\begin{align*}
		\operatorname{M}^{\operatorname{NA}}(\mathcal{X}_0, \mathcal{L}_0)+\frac{S}{V_0(n+1)!}(\mathcal{L}_0^{n+1}) -\sum_{i=1}^m\frac{1}{V_i(n+1)!}(\mathcal{L}_{i}^{n+1})+\frac{1}{V_0n!}\sum_{i=1}^m(\mathcal{L}_i\cdot\mathcal{L}_0^n)\geq 0.
	\end{align*}
\end{theorem}
The inequality (\ref{cohomologicalcoupledKsemistable}) indicates a semistability condition which generalizes \cite[Definition 3.11]{boucksom2017uniform}, \cite[Deﬁnition 2.22 (i)]{DR} and \cite[Deﬁnition 3.8]{sjostrom2018k} in the coupled case.
\begin{definition}\label{coupledKsemistableDef}
	Let $(X,([\omega_i])_{i=0}^m)$ be a compact K\"ahler manifold with a tuple of K\"ahler classes.\\
	$(\mathrm i)$ Let $(\mathcal{X}_i,\mathcal{A}_i)_{i=0}^m$ be a tuple of (compactified) relatively K\"{a}hler cohomological test conﬁgurations of $(X,\alpha_i),i=0,\cdots,m$, the coupled Donaldson–Futaki invariant of $(\mathcal{X}_i,\mathcal{A}_i)_{i=0}^m$ is defined by\begin{small}
		\begin{align*}
		\operatorname{DF}_c((\mathcal{X}_i,\mathcal{A}_i)_{i=0}^m):=\operatorname{DF}(\mathcal{X}_0, \mathcal{A}_0)+\frac{S}{V_0(n+1)!}(\mathcal{A}_0^{n+1}) -\sum_{i=1}^m\frac{1}{V_i(n+1)!}(\mathcal{A}_{i}^{n+1})+\frac{1}{V_0n!}\sum_{i=1}^m(\mathcal{A}_i\cdot\mathcal{A}_0^n).
	\end{align*}
	\end{small}Here the intersection number and classical DF-invariant are defined on a common smooth equivariant resolution (see \cite[Section 3.1 nd 3.3]{sjostrom2018k}), which is well-defined by the projection formula.\\
	$(\mathrm{ii})$ We say that $(X,([\omega_i])_{i=0}^m)$ is coupled K-semistable if for any tuple of  (compactified) relatively K\"{a}hler cohomological test conﬁgurations $(\mathcal{X}_i,\mathcal{A}_i)_{i=0}^m$ of $(X,\alpha_i),i=0,\cdots,m$, we have $\operatorname{DF}_c((\mathcal{X}_i,\mathcal{A}_i)_{i=0}^m)\geq 0$.\\
	\indent Suppose now $[\omega_i]=c_1(L_i)$ for ample line bundles $L_i$.\\
	$(\mathrm i^\prime)$ Let $(\mathcal{X}_i,\mathcal{L}_i)_{i=0}^m$ be a tuple of (compactified) ample alegbraic test conﬁgurations of $(X,L_i),i=0,\cdots,m$, the coupled Donaldson–Futaki invariant of $(\mathcal{X}_i,\mathcal{L}_i)_{i=0}^m$ is defined by\begin{small}
		\begin{align*}
			\operatorname{DF}_c((\mathcal{X}_i,\mathcal{L}_i)_{i=0}^m):=\operatorname{DF}(\mathcal{X}_0, \mathcal{L}_0)+\frac{S}{V_0(n+1)!}(\mathcal{L}_0^{n+1}) -\sum_{i=1}^m\frac{1}{V_i(n+1)!}(\mathcal{L}_{i}^{n+1})+\frac{1}{V_0n!}\sum_{i=1}^m(\mathcal{L}_i\cdot\mathcal{L}_0^n).
		\end{align*}
	\end{small}$(\mathrm{ii}^\prime)$ We say that $(X,(L_i)_{i=0}^m)$ is coupled algebraically K-semistable if for any tuple of (compactified) ample algebraic test conﬁgurations $(\mathcal{X}_i,\mathcal{L}_i)_{i=0}^m$ of $(X,L_i),i=0,\cdots,m$, we have $\operatorname{DF}_c((\mathcal{X}_i,\mathcal{L}_i)_{i=0}^m)\geq 0$.
\end{definition}\noindent Moreover, we can show the following extension of \cite[Theorem A (i), Proposition 3.12 and Proposition 3.14]{sjostrom2018k}.
\begin{theorem}\label{coupledKsemistable}\mbox{}\\
	$(\mathrm i)$ If $(X,([\omega_i])_{i=0}^m)$ admits a tuple of coupled cscK metrics, then $(X,([\omega_i])_{i=0}^m)$ is coupled K-semistable.\\
	$(\mathrm{ii})$ $(X,(L_i)_{i=0}^m)$ is coupled algebraically K-semistable if and only if $(X,(c_1(L_i))_{i=0}^m)$ is coupled K-semistable.
\end{theorem}
Finally, in preparation for the second paper \cite{toappear} in this series, we also extend $\mathcal{M}_{\operatorname{ccscK}}$ to the finite energy class and prove its convexity along tuples of finite energy geodesics, extending \cite[Theorem 1.2]{BDL2017GT} to the coupled case (for $\chi=0$).
	\begin{theorem}\label{finiteenergyextension}
		Suppose $X$ is a compact connected K\"ahler manifold with K\"{a}hler forms $\omega_0,\ \cdots,\ \omega_m$. The coupled K-energy (\ref{coupledMabuchiDef}) can be extended to a functional $\mathcal{M}_{\operatorname{ccscK}}:\prod_{i=0}^m\mathcal{E}^1(X,\omega_i)\to \mathbb{R}\cup\{\infty\}$ 
		using the formula \begin{equation}\label{finiteenergyextensionformula}
			\begin{aligned}
				&\mathcal{M}_{\operatorname{ccscK}}(u_0,\cdots,u_m)= \mathcal{M}_{\operatorname{cscK}}(u_0)+\frac{\overline{S}+n}{V_0n!}\operatorname{E}_{\omega_0}(u_0)\\ &\quad \sum_{i=1}^m\left(\frac{1}{V_0n!}\int_Xu_i(\omega_0+dd^cu_0)^n-\frac{1}{V_in!}\operatorname{E}_{\omega_i}(u_i)+\frac{1}{V_0(n-1)!}\operatorname{E}^{\omega_i}_{\omega_0}(u_0)\right).
			\end{aligned}
		\end{equation}Thus extended, $\mathcal{M}_{\operatorname{ccscK}}
		$ is the greatest $\prod_{i=0}^m\ d_{1,\omega_i}$-lsc extension of (\ref{coupledMabuchiDef}). Additionally, $\mathcal{M}_{\operatorname{ccscK}}
		$ is convex along any tuple of finite-energy geodesics in $\prod_{i=0}^m\mathcal{E}^1(X,\omega_i)$.
	
\end{theorem}

\paragraph{Organization of the paper} In section \ref{sec:2}, we recall essential background material, including geodesics in the space of K\"ahler metrics, test configurations, finite energy classes, and the basics of coupled KE and cscK equations. 
In section \ref{sec:3} we introduce three energy functionals
$\mathcal{M}_{\operatorname{ccscK}}$ which integrates the moment map, $\mathcal{M}'$ which is adapted to a Donaldson-Futaki invariant
used in the papers of Hultgren-Witt Nystr\"om and Datar-Pingali, and $\mathcal{N}_{\operatorname{cKE}}$ which was used by Hultgren-Witt Nystr\"om. {We present the Chen-Tian formulas for $\mathcal{M}_{\operatorname{ccscK}}$ and $\mathcal{M}'$, and calculate their slopes at infinity along smooth geodesic rays compatible with tuples of relatively K\"ahler cohomological and ample algebraic test configurations. We also demonstrate the convexity of $\mathcal{M}_{\operatorname{ccscK}}$ along weak geodesics, with the help of which we prove Theorem \ref{our6}. The functional $\mathcal{M}_{\operatorname{ccscK}}$ is further extended to the finite energy class and shown to be convex along tuples of finite energy geodesics. In section \ref{sec:4} we compare the three functionals $\mathcal{M}_{\operatorname{ccscK}},\ \mathcal{M}'$ and $\mathcal{N}_{\operatorname{cKE}}$, obtaining Theorems \ref{ourt2} and Corollary \ref{our4}.

\paragraph{Disclosure on the use of AI}
The authors used ChatGPT (GPT-6 Astra) to help identify suitable lemmas from the literature for the proofs of Proposition \ref{BigSymmetryProp} and Lemmas \ref{extensionLemma1} and \ref{extensionLemma2}, and take full responsibility for the correctness of their application. The authors also used ChatGPT (GPT-6 Astra, DeepL) for language editing and grammar checking.
\section{Background}
\label{sec:2}
In this section, we recall some basic setups of the ccscK and cKE equations from \cite{datar2020coupled} and \cite{hultgren2019coupled,Takahashi} respectively. To set the stage, we begin by briefly reviewing some fundamental concepts in the classical setting.

\subsection{Geodesic rays and test configurations}
\label{subsec:1}
The natural domain for solving these equations (cscK, KE, ccscK, and cKE) on a compact K\"ahler manifold $X$ is the space of K\"ahler metrics $\mathcal{K}(\alpha)$ in a fixed K\"aler class $\alpha$ or a product of such spaces. Fixing a reference metric $\omega\in\alpha$, this space can be identified with $\mathcal{H}(\omega)/\mathbb{R}$, where $\mathcal{H}(\omega)$
is the space of smooth strictly $\omega$-plurisubharmonic potentials.

$\mathcal{H}(\omega)$ naturally carries a Riemannian metric $d_2$ defined for $\psi^1,\psi^2 \in T_\varphi\mathcal{H}(\omega)\cong C^\infty(X)$ by (c.f. \cite{Mabuchi1987})\begin{align*}
	\braket{\psi^1,\psi^2}_{\varphi}:=\frac{1}{[\omega]^n}\int_X\psi^1\psi^2(\omega+dd^c\varphi)^n
\end{align*}with geodesic equation for a curve $\varphi(t,x)$, $t\in [0,s]$, in $(\mathcal{H}(\omega),\braket{\cdot,\cdot})$ (c.f. \cite{Sem92})\begin{align*}
		(\pi_1^*\omega+dd^c_{(x,\tau)}\Phi)^{n+1}=0
	\end{align*}where $\Phi:X\times \{\tau: e^{-s}\leq |\tau|\leq 1\}\to \mathbb{R}$, $\Phi(x,\tau):=\varphi(-\ln|\tau|,x)$ and $\pi_1:X\times \{\tau: e^{-s}\leq |\tau|\leq 1\}\to X$ is the projection. Unfortunately, this smooth PSH space is not geodesically convex (see \cite{lempert2013geodesics}), which presents a major obstacle to employing the variational approach. However, by extending the domain to the space\begin{align*}
	\mathcal{H}^{1,\bar{1}}(\omega):=\{\varphi\in \operatorname{PSH}(\omega):dd^c\varphi\in L^\infty_{\operatorname{loc}}\},
\end{align*}any two smooth $\omega$-psh potentials can be connected by a unique weak geodesic segment $\varphi(t)$, in the sense of the following well known theorem:
\begin{theorem}[\hspace*{-0.6ex}\text{\cite{Chen2000JDG}, \cite{Błocki2012}}]
	We say $\varphi(t),t\in [0,s]$ (or $t\in [0,\infty)$) in $\operatorname{PSH}(X,\omega)\cap L^\infty_{\operatorname{loc}}$ is a subgeodesic segment (ray) if $\pi^*_1\omega+dd^c_{(x,\tau)}\Phi\geq 0$ as current. If further $(\pi^*_1\omega+dd^c_{(x,\tau)}\Phi)^{n+1}=0$ in sense of Bedford–Taylor \cite{BedfordTaylor76,BedfordTaylor82}, then we say $\varphi(t),t\in [0,s]$ (or $t\in [0,\infty)$) is a weak geodesic segment (ray). For any two smooth $\omega$-psh potentials $\varphi$ and $\varphi'$, there exists a continuous weak geodesic segment $\varphi(t):t\in [0,1]$ such that $\varphi(0)=\varphi,\varphi(1)=\varphi'$ and $dd^c_{(x,\tau)}\Phi\in L^\infty_{\operatorname{loc}}$.
\end{theorem}

Another key tool in classical KE and cscK theory is the notion of a test configuration, which is given by Donaldson \cite{donaldson2002scalar}. We will follow \cite{boucksom2019} for the algebraic case and \cite{sjostrom2018k} for a generalization to the cohomological case. For convenience we will use the compactified version.

\begin{definition}Let $(V,L)$ be a smooth projective complex variety coupled with a line bundle. An algebraic test configuration $(\mathcal{V},\mathcal{L})$ for $(V,L)$ consists of the following data:\\
	(i) a flat, projective morphism of normal schemes $\pi: \mathcal{V} \to \mathbb{P}^1$ which is $\mathbb{C}^*$-equivariant;\\
	(ii) a $\mathbb{C}^*$-equivariant $\mathbb{Q}$-line bundle $\mathcal{L}$;\\
	(iii) $(\mathcal{V}_t,\mathcal{L}_t) \simeq (V,L)$ for $t\neq 0$.
	
	We say that an algebraic test configuration $(\mathcal{X}, \mathcal{L})$ is smooth if the total space $\mathcal{X}$ is smooth, and ample (resp. semiample) if $\mathcal{L}$ is relatively ample (resp. relatively semiample).
\end{definition}

\begin{definition}[\hspace*{-0.6ex}\text{\cite[Definition 3.2, 3.3 \& 3.4]{sjostrom2018k}}]
	Let $X$ be a compact K\"{a}hler manifold and $\alpha\in H^{1,1}(X,\mathbb{R})$ be a cohomology class on $X$. A cohomological test configuration $(\mathcal{X},\mathcal{A})$ of $(X,\alpha)$ consists of:\\
	(i) a normal compact K\"{a}hler complex space $\mathcal{X}$ with a flat morphism $\pi:\mathcal{X}\to\mathbb{P}^1$ which is $\mathbb{C}^*$-equivariant;\\
	(ii) a $\mathbb{C}^*$-equivariant isomorphism $\mathcal{X}\backslash\mathcal{X}_0\simeq X\times(\mathbb{P}^1\backslash\{0\})$;\\
	(iii) a $\mathbb{C}^*$-invariant Bott-Chern cohomology class $\mathcal{A}\in H^{1,1}_{\operatorname{Bott-Chern}}(\mathcal X,\mathbb R)$ such that the pullback of $\mathcal{A}$ under the isomorphism $\mathcal{X}\backslash\mathcal{X}_0\simeq X\times(\mathbb{P}^1\backslash\{0\})$ is $\operatorname{pr}_1^*\alpha$.

	We say that a cohomological test configuration $(\mathcal{X}, \mathcal{A})$ for $(X, \alpha)$ is smooth
	if the total space $\mathcal{X}$ is smooth. In case $\alpha \in H^{1,1}(X,\mathbb{R})$ is Kähler,
	we say that $(\mathcal{X}, \mathcal{A})$ is relatively Kähler if the cohomology class 
	$\mathcal{A}$ is relatively Kähler, i.e.\ there is a Kähler form $\beta$ on $\mathbb{P}^1$ such that $\mathcal{A} + \pi^{*}\beta$  is Kähler on $\mathcal{X}$.
\end{definition}

\begin{definition}[\hspace*{-0.6ex}\text{\cite[Section 2]{boucksom2017uniform} and \cite[Section 3.1]{sjostrom2018k}}]
	For cohomological (or algebraic) test configurations $(\mathcal{X},\mathcal{A})$ and $(\mathcal{X}',\mathcal{A}')$ (or $(\mathcal{X},\mathcal{L})$ and $(\mathcal{X}',\mathcal{L}')$), the isomorphism $\mathcal{X}_1\simeq X\simeq \mathcal{X}_1'$ induces a canonical $\mathbb{C}^*$-isomorphism $\mathcal{X}'\backslash\mathcal{X}_0'\overset{\simeq}{\to} \mathcal{X}\backslash\mathcal{X}_0$. We say that $\mathcal{X}'$ dominates $\mathcal{X}$ if this isomorphism extends to a morphism $\mathcal{X}' \to \mathcal{X}$.
\end{definition}
\begin{proposition}[\hspace*{-0.6ex}\text{\cite[Section 2.6]{boucksom2017uniform} and \cite[Proposition 3.10]{sjostrom2018k}}]
		Let $(\mathcal{X},\mathcal{A})$ be a smooth cohomological test configuration for $(X,\alpha)$ dominating $X\times\mathbb{P}^1$, with $\mu:\mathcal{X}\to X\times\mathbb{P}^1$ the corresponding canonical $\mathbb{C}^*$-equivariant bimeromorphic morphism. Then there exists a unique $\mathbb{R}$-divisor $D$ supported on the central fiber $\mathcal{X}_0$ such that
		\begin{align*}
			\mathcal{A}=\mu^*\operatorname{pr}_1^*\alpha+[D]
		\end{align*}
		in $H^{1,1}(\mathcal{X},\mathbb{R})$.
\end{proposition}
For a smooth ample algebraic test configuration $(\mathcal{X},\mathcal{L})$, a strictly positive $\mathbb{S}^1$-invariant metric $\Phi$ near the central fiber naturally induces an $\mathcal{H}(\omega)$-subgeodesic ray. In the cohomological case, this notion is extended as follows:
	\begin{definition}[\hspace*{-0.6ex}\text{\cite[Section 4.1, 4.2 and Lemma 4.10]{sjostrom2018k}}]\label{compatible}
		Let $(\mathcal{X},\mathcal{A})$ be a smooth cohomological test configuration for $(X,\alpha)$ dominating $X\times\mathbb{P}^1$, with corresponding bimeromorphic morphism $\mu:\mathcal{X}\to X\times\mathbb{P}^1$, and write $\mathcal{A}=\mu^*\operatorname{pr}_1^*\alpha+[D]$. Choose a smooth $\mathbb{S}^1$-invariant closed real $(1,1)$-form $\Theta$ representing $\mathcal{A}$ and a smooth closed real $(1,1)$-form $\theta$ representing $\alpha$. Set $\theta_D:=\Theta-\mu^*\operatorname{pr}_1^*\theta$, and choose an $\mathbb{S}^1$-invariant locally integrable function $\psi_D$ satisfying $dd^c\psi_D+\theta_D=\delta_D$, where $\delta_D$ denotes the current of integration along $D$. For a ray $(\varphi_t)_{t\geq 0}$, let $\Phi$ denote the associated $\mathbb{S}^1$-invariant function defined by $\Phi(x,e^{-t+\sqrt{-1}s})=\varphi_t(x)$.\\
			(i) A smooth ray $(\varphi_t)_{t\geq 0}$ in $\mathcal{C}^{\infty}(X)$ is said to be $\mathcal{C}^{\infty}$-compatible with $(\mathcal{X},\mathcal{A})$ if there exists an $\mathbb{S}^1$-invariant function $\Psi\in\mathcal{C}^{\infty}(\pi^{-1}(\Delta))$ such that $\Psi=\Phi\circ\mu+\psi_D$ on $\pi^{-1}(\Delta\backslash\{0\})$. Here $\Delta$ is the unit disc in $\mathbb{C}$.\\
			(ii) Suppose that $\theta=\omega$ is a K\"ahler form. A ray $(\varphi_t)_{t\geq 0}$ in $\operatorname{PSH}(X,\omega)\cap L^\infty(X)$ is said to be $\mathcal{C}^{1,\bar{1}}$-compatible (resp. $L^\infty$-compatible) with $(\mathcal{X},\mathcal{A})$ if there exists an $\mathbb{S}^1$-invariant function $\Psi\in\mathcal{C}^{1,\bar{1}}(\pi^{-1}(\Delta))$ (resp. $\Psi\in L^\infty_{\mathrm{loc}}(\pi^{-1}(\Delta))$) such that $\Psi=\Phi\circ\mu+\psi_D$ on $\pi^{-1}(\Delta\backslash\{0\})$. Here $\mathcal{C}^{1,\bar{1}}$ means that $dd^c\Psi$ has locally bounded coefficients.
	\end{definition}
	Note that Definition \ref{compatible} is independent of the choice of $\psi_D$, $\Theta$ and $\theta$, since replacing them by $\psi_D'$, $\Theta+dd^ch$ and $\theta+dd^cg$ respectively (with $g,h$ smooth and $h$ being $\mathbb{S}^1$-invariant) changes $\Psi$ only by the addition of a smooth function of the form $\psi_D'-\psi_D=g\circ\operatorname{pr}_1\circ\mu-h+C$, where $C$ is a constant.
	\begin{lemma}[\hspace*{-0.6ex}\text{\cite[Lemma 4.4]{sjostrom2018k}}]\label{CompatibleRayExists}
		If $(\mathcal{X},\mathcal{A})$ is a smooth, relatively K\"ahler cohomological test configuration of $(X,\omega)$, then there is a subgeodesic ray $\mathcal{C}^\infty$-compatible with $(\mathcal{X},\mathcal{A})$.
	\end{lemma}
	\subsection{Finite energy classes}
	As a generalization of the usual wedge product, for bounded (may not smooth) psh potentials $u_i\in \operatorname{PSH}(X,\omega_i)\cap L^\infty(X), i=1,\cdots n$, Bedford--Taylor introduced the mixed Monge-Amp\`{e}re measures $(\omega_1+dd^cu_1)\wedge\cdots\wedge (\omega_n+dd^cu_n)$ (see \cite{BedfordTaylor76,BedfordTaylor82}). This measure is symmetric in the pairs $(\omega_i,u_i)$ and has total mass $\int_X\omega_1\wedge\cdots\wedge\omega_n$. Later in \cite{GZ07,BBEGZ10}, for possibly unbounded case, the non-pluripolar product for $u_i\in \operatorname{PSH}(X,\omega_i), i=1,\cdots n$ is defined by\begin{align*}
		&(\omega_1+dd^cu_1)\wedge\cdots\wedge(\omega_n+dd^cu_n)\\&:=\lim_{k\to\infty}\mathsf{1}_{\cap_{i=1}^n\{u_j>-k\}}(\omega_1+dd^c\max\{u_1,-k\})\wedge\cdots\wedge(\omega_n+dd^c\max\{u_n,-k\}).
	\end{align*}Both these generalizations of the wedge product admit corresponding theories for solving complex Monge-Amp\`ere equations, under suitable full-mass or finite-energy assumptions, extending Yau's theorem. For the Bedford--Taylor product, see \cite{Kolodziej98} for existence, \cite{BedfordTaylor82,Kolodziej98} for uniqueness, and \cite{Kolodziej08} for regularity. For the non-pluripolar product, see \cite{GZ07,BBEGZ10} for existence, \cite{Dinew2009,BBEGZ10} for uniqueness and \cite{BBGZ} for the variational approach. In this section, we recall some basic results from finite-energy
	pluripotential theory, focusing on the metric and topology that will be useful later in the paper.
	
	In contrast with the bounded case, the non-pluripolar product may have strictly smaller total mass than $\int_X\omega_1\wedge\cdots\wedge\omega_n.$ Following \cite{GZ07}, the class $\mathcal{E}(X,\omega)$ of potentials with full volume is defined by\begin{align*}
	\mathcal{E}(X,\omega):=\left\{u\in \operatorname{PSH}(X,\omega):\int_X(\omega+dd^cu)^n=\int_X\omega\right\},
	\end{align*}and the corresponding ﬁnite-energy classes\begin{align*}
	\mathcal{E}^p(X,\omega):=\left\{u\in \mathcal{E}(X,\omega):u\in L^p((\omega+dd^cu)^n)\right\}.
	\end{align*}When $p=2$, the finite energy class $\mathcal{E}^2(X,\omega)$ can be
	realized as the metric completion of $\mathcal{H}_\omega$ with respect to
	the $d_2$. In fact, as shown in \cite{Darvas2}, one has the following theorem:
	\begin{theorem}[\hspace*{-0.6ex}\text{\cite[Theorem 2]{Darvas2}}]
		For $u\in\mathcal{H}_\omega$ and $\xi\in T_u\mathcal{H}_\omega\simeq C^\infty(X)$, define the $L^p$-Finsler norm by \begin{align*} \|\xi\|_{p,u} := \left( \frac{1}{V}\int_X|\xi|^p\omega_u^n \right)^{1/p}. \end{align*} The associated path-length metric on $\mathcal{H}_\omega$ is \begin{align*} d_p(u_0,u_1) := \inf_{\gamma} \int_0^1\|\dot{\gamma}_t\|_{p,\gamma_t}\,dt, \end{align*} where the infimum is taken over all smooth curves $\gamma\colon[0,1]\to\mathcal{H}_\omega$ joining $u_0$ to $u_1$.
		
		Then the metric completion of $(\mathcal{H}_\omega,d_p)$ equals $(\mathcal{E}^p(X,\omega),d_p)$, where \begin{align*} d_p(u_0,u_1) := \lim_{k\to\infty}d_p(u_0^k,u_1^k), \end{align*} for any smooth decreasing sequences $ \{u_i^k\}_{k\in\mathbb{N}}\subset\mathcal{H}_\omega$ converging pointwise to $u_i\in\mathcal{E}^p(X,\omega)$, $i=0,1$. Moreover, for each $t\in(0,1)$, define \begin{align*} u_t:=\lim_{k\to\infty}u_t^k, \end{align*} where $u_t^k$ is the weak geodesic connecting $u_0^k$ and $u_1^k$. Then $u_t\in\mathcal{E}^p(X,\omega)$, the curve $ [0,1]\ni t\mapsto u_t\in\mathcal{E}^p(X,\omega)$ is well-defined independently of the choices of approximating sequences, and this curve is a $d_p$-geodesic. More precisely, the assertion that this curve is a $d_p$-geodesic means that it has constant speed: 
		\begin{align*} 
			d_p(u_s,u_t) = |t-s|d_p(u_0,u_1), \qquad s,t\in[0,1]. 
		\end{align*}
	\end{theorem}
	\noindent The growth of the $d_p$-distance can be controlled by an explicit
	functional.
	\begin{theorem}[\hspace*{-0.6ex}\text{\cite[Theorem 3]{Darvas2}}]\label{Theorem3ofDarvas2}
		For a K\"ahler form $\omega$ and $p\geq 1$, set \begin{align*} \mathbf{I}_{p,\omega}(u,v) :=\left(\int_X |u-v|^p\omega_u^n\right)^{1/p}+\left(\int_X |u-v|^p\omega_v^n\right)^{1/p}, \end{align*}
		where $\omega_u:=\omega+dd^cu$. Then there exists $C > 1$ such that \begin{align*} \frac{1}{C}\mathbf{I}_{p,\omega}(u,v) \leq d_{p,\omega}(u,v)\leq C\mathbf{I}_{p,\omega}(u,v),
		\end{align*}where $d_{p,\omega}$ is the $d_p$-metric on $\mathcal{E}^p(X,\omega)$.
	\end{theorem}
	We next recall two classical functionals in K\"ahler geometry. Both of them extend naturally to $\mathcal{E}^1(X,\omega)$.
	\begin{example}\label{MAfunctional}
		The Monge-Amp\`ere energy functional is defined by\begin{align*}
			&\operatorname{E}_\omega:\mathcal{E}^1(X,\omega)\to \mathbb{R},\\
			&\operatorname{E}_\omega(\varphi):=\frac{1}{n+1}\sum_{j=0}^n\int_X\varphi(\omega+dd^c\varphi)^{n-j}\wedge\omega^j,
		\end{align*}
		and the contracted Monge-Amp\`ere energy functional\begin{align*}
		&\operatorname{E}_\omega^\chi:\mathcal{E}^1(X,\omega)\to \mathbb{R},\\
		&\operatorname{E}_\omega^\chi(\varphi):=\frac{1}{n}\sum_{j=0}^{n-1}\int_X\varphi\chi\wedge(\omega+dd^c\varphi)^{n-j-1}\wedge\omega^j,\end{align*}where $\chi$ is a closed smooth $(1,1)$-form. By \cite[Proposition 4.1]{BDL2017GT} and \cite{BBEGZ10}, $\operatorname{E}_\omega$ is $d_1$-continuous. The same holds for its contracted version.
		
	\end{example}
	\begin{example}
			The Aubin $I$-functional is defined by\begin{align*}
				I_{\omega}&:\mathcal{E}^1(X,\omega)\times\mathcal{E}^1(X,\omega)\to \mathbb{R},\\
				I_{\omega}(u_0,u_1)&:=\int_X(u_0-u_1)(\omega_{u_1}^n-\omega_{u_0}^n), \ \text{where}\  \omega_{u_i}=\omega+dd^cu_i.
			\end{align*}One can prove that $I_\omega$ is non-negative. It follows from Theorem \ref{Theorem3ofDarvas2} that $d_1$-convergence implies convergence with respect to $I_\omega$. In fact, the functional $I_\omega$ was used in \cite[Definition 5.5]{BBGZ} to define convergence in energy, thus playing a distance-like role, before the introduction of $d_1$ in \cite{Darvas2}.
	\end{example}

	We conclude this section with an estimate that will be useful for comparing energy functionals associated with different K\"ahler forms. Following \cite[Definition 2.5]{BBGZ}, for $C>0$, set
	\begin{align}\label{E_C}
		\mathcal{E}_C^1(X,\omega)
		:=\left\{u\in\mathcal{E}^1(X,\omega):\sup_Xu\leq 0,\frac{1}{V_\omega}\operatorname{E}_\omega(u)\geq-C\right\}.
	\end{align}Then one has the following estimate:
	\begin{proposition}[\hspace*{-0.6ex}{\cite[Lemma 3.13]{BBGZ}}]\label{BBGZLemma3.13}
		For every $C>0$, there exists a constant $A_C>0$ such that
		\begin{align*}
			\sup_{u\in\mathcal{E}_C^1(X,\omega)}
			\left|\frac{1}{V_\omega}\int_Xu\omega_{\psi_1}^n-\frac{1}{V_\omega}\int_Xu\omega_{\psi_2}^n\right|
			\leq
			A_C\left(\frac{I_\omega(\psi_1,\psi_2)}{V_\omega}\right)^{1/2}
		\end{align*}for all $\psi_1,\psi_2\in\mathcal{E}_C^1(X,\omega)$.
	\end{proposition}

	
	\subsection{Coupled KE equation and coupled cscK equation}
	\label{subsec:2}
	{Let $X$ be an $n$-dimensional compact complex manifold, with $\alpha_i,i=0,...,m$, being K\"{a}hler classes on it. Given K\"ahler forms $\omega_i$ in the classes $\alpha_i$ for $i=0,\dots,m$, we say that the K\"{a}hler forms $(\omega_i)_{i=0}^m$ satisfies the cKE equation if (\ref{cKE}) holds, and that $(\omega_i)_{i=0}^m$ satisfy the ccscK system if (\ref{ccscK}) holds. By the Bianchi identity, ccscK equation (\ref{ccscK}) is equivalent to cKE equation (\ref{cKE1}) for the case of $\epsilon=1$, under the assumption that $c_1(X)=\sum_{i=0}^m[\omega_i]$.}
	
	In \cite{datar2020coupled}, an infinite-dimensional moment map interpretation for the ccscK equation was introduced, in the spirit of \cite{fujiki1992moduli,Donaldson1997Fields}. To explain this process, {assuming that $\alpha_i=c_1(L_i)$, for $L_i,i=0,...,m$ being ample line bundles on $X$.} Let us momentarily forget the holomorphic structure on $L_i$, and let $L_i'$ be the smooth complex line bundle that underlies $L_i$. Let $F(L_i')$ be the frame bundle of $L_i',i=1,...,m$ (not including $i=0$), which is a principal $\mathbb{C}^*$-bundle. Then $E:=\bigoplus_{i=1}^mF(L_i')$ is a principal $(\mathbb{C}^*)^m$-bundle. 
	
	Now suppose we are given a $(\mathbb{S}^1)^m$-reduction $H=(h_1,...,h_m)$ on $E$ and a symplectic form $\omega_0\in c_1(L_0)$ on $X$. Note that providing an $(\mathbb{S}^1)^m$-reduction $H$ on $E$ is equivalent to providing Hermitian metrics $h_1,...,h_m$ on $L_1,...,L_m$ respectively. Let $\mathcal{G}_H$ be the gauge group of princpal $(\mathbb{S}^1)^m$-bundle $E_H$, and $\mathcal{H}$ be the group of Hamiltonian symplectomorphisms on $(X,\omega_0)$. Then by \cite[Section 1.3]{alvarez2013coupled}, we have an exact sequence of Lie groups\begin{align}
		0\to \mathcal{G}_H\to \tilde{\mathcal{G}}_H\to \mathcal{H}\to 0,
	\end{align}where $\tilde{\mathcal{G}}_H$ is called the extended gauge group, consisting of automorphisms $g=(g_1,...,g_m)$ of $E$ that cover elements of $\mathcal{H}$.
	
	$\tilde{\mathcal{G}}_H$ actions on the space $\mathcal{J}_{\omega_0}\times \mathcal{A}_{H}$ naturally, in which $\mathcal{J}_{\omega_0}$ is the space of almost complex structures compatible with and taming $\omega_0$, and $\mathcal{A}_{H}$ is the space of connections on $E_H$ (equivalent to the space of $m$-tuples of $h_i$-unitary connections on $L_i'$). The action is given by\begin{align*}
		g.(J,A_1,...,A_m):=(\check{g}_*J, g_{1*}A_1,...,g_{m*}A_m),
	\end{align*}where $\check{g}$ is the Hamiltonian symplectomorphism covered by $g$, and $A_i$ should be considered as $(\operatorname{Lie}\mathbb{S}^1)$-valued $1$-form on $F(L_i')$.
	
	Unlike the K\"{a}hler-Yang-Mills case \cite{alvarez2013coupled}, for the ccscK equation, the base space of action is a subspace $\mathcal{M}$ of $\mathcal{J}_{\omega_0}\times \mathcal{A}_{H}$ defined by\begin{align*}
		\mathcal{M}:=\left\{(J,A_1,...,A_m):J\text{ is integrable and }\omega_{A_i}:=\frac{\sqrt{-1}}{2\pi}F_{A_i}\text{ is positive }(1,1)\text{-form}\right\},
	\end{align*}with symplectic form $\Omega_{\mathcal{M}}$ with
	\begin{align*}
		&\Omega_{\mathcal{M}}((S,\sqrt{-1}a_1,...,\sqrt{-1}a_m),(T,\sqrt{-1}b_1,...,\sqrt{-1}b_m))\\:= & \ -\sum_{i=1}^m\frac{V_0}{4\pi^2V_i}\int_Xa_i\wedge b_i\wedge\frac{\omega_{A_i}^{n-1}}{(n-1)!}-\frac{1}{2}\int_X\operatorname{tr}(JST)\frac{\omega_0^n}{n!},
	\end{align*}where the tangent vectors are considered as 
	\begin{align*}
		&(S,\sqrt{-1}a_1,...,\sqrt{-1}a_m),(T,\sqrt{-1}b_1,...,\sqrt{-1}b_m)\in T_{(J,A_1,...,A_m)}(\mathcal{J}_{\omega_0}\times \mathcal{A}_{H})\\&\simeq \{S\in \operatorname{End}(TX):SJ+JS=0,\omega_0(SY,JZ)=\omega_0(SZ,JY)\}\times \left(\Lambda^1(X,\sqrt{-1}\mathbb{R})\right)^m.
	\end{align*}Actually $\Omega_{\mathcal{M}}(\cdot,\textbf{I}\cdot)$ deﬁnes a Riemannian metric on $\mathcal{M}$ (but not on $\mathcal{J}_{\omega_0}\times \mathcal{A}_{H}$ because we need $\omega_{A_i}$ to be positive) for the integrable complex structure $\textbf{I}_{(J,A)}(\sqrt{-1}a,S):=(JS,-\sqrt{-1}a(J\cdot))$ on $\mathcal{J}_{\omega_0}\times \mathcal{A}_{H}$ (see \cite[Remark 2.3]{alvarez2013coupled}).
	
	Now we are ready for the infinite-dimensional moment map interpretation for ccscK equation, which will be helpful to understand the coupled K-energy in Section \ref{sec:3}.
	\begin{theorem}[\hspace*{-0.6ex}\text{\cite[Theorem 8]{datar2020coupled}}]
		The action of $\tilde{\mathcal{G}}_H$ restricts to a Hamiltonian action on $(\mathcal{M},\Omega)$ with moment map
		\begin{equation}\label{Futaki}
			\begin{aligned}
				\mu_{(J,A)}(\tilde{\zeta}) =\sum_{i=1}^m\int_M H_{\tilde{\zeta}_i} \left(\frac{V_0}{V_i} \frac{\omega_i^n}{n!} - \frac{\omega_0^n}{n!} \right)+ \int_M H_{\zeta} \left( (S_{\omega_0,J} - \hat S) \frac{\omega_0^n}{n!} - \sum_{i=1}^m\omega_i \wedge \frac{\omega_0^{n-1}}{(n-1)!} \right),
			\end{aligned}
		\end{equation}where $\tilde{\zeta}=(\tilde{\zeta}_i)\in \operatorname{Lie}\tilde{\mathcal{G}}_H$ is a right invariant vector field on $E$ covering Hamiltonian vector field $\zeta$, $H_{\tilde{\zeta}_i}=\frac{\sqrt{-1}}{2\pi}A_i(\tilde{\zeta}_i)$ and $H_{\zeta}$ is the Hamiltonian of $\zeta$ with respect to $\omega_0$, defined by $dH_{\zeta}=-\iota_{\zeta}\omega_0$.
	\end{theorem}
	\begin{corollary}[\hspace*{-0.6ex}\text{\cite[Theorem 2]{datar2020coupled}}]Given symplectic form $\omega_0$ on $X$ and $(\mathbb{S}^1)^m$-reduction $H$ on $E$, then for $(J,A):=(J,A_1,...,A_m)\in \mathcal{M}$, $\mu_{(J,A)}=0$ if and only if $\omega_0$ and $\omega_i=\frac{\sqrt{-1}}{2\pi}F_{A_i}$ are ccscK metrics.
	\end{corollary}
	\begin{remark}
		There is another approach to interpreting a moment map of the ccscK equation (\ref{ccscK}) that does not use the gauge group. These ideas are introduced by Lee \cite{Lee2023AJM}. He also provides an explicit expression for the coupled K-energy and demonstrates convexity along smooth geodesics. Such expression will be used in Section \ref{sec:3}.
	\end{remark}
	
	\boldmath
	\section{Coupled energy functionals.}
	\unboldmath
	\label{sec:3}
	\subsection{Coupled K-energy}
	Let us recall the coupled K-energy suggested in \cite[Remark 12]{datar2020coupled}. 
	\begin{definition}\label{coupledMabuchi}The coupled K-energy (or coupled Mabuchi energy) $\mathcal{M}_{\operatorname{ccscK}}:\prod_{i=0}^{m}\mathcal{K}([\omega_i])\to \mathbb{R}$ is defined by\begin{equation}\label{coupledMabuchiDef}
			\begin{aligned}
			&\mathcal{M}_{\operatorname{ccscK}}((\omega_i(s)),(\omega_i(0))) :=-\sum_{i=1}^m\int_0^s\int_X\dot{\varphi_i}(t)\left(\frac{1}{V_i} \frac{\omega_i(t)^n}{n!}- \frac{1}{V_0}\frac{\omega_0(t)^n}{n!} \right)\wedge dt \\& - \frac{1}{V_0}\int_0^s\int_M \dot{\varphi_0}(t) \left((S_{\omega_0(t),J}-\hat{S})\frac{\omega_0(t)^n}{n!} - \sum_{i=0}^m\omega_i(t) \wedge \frac{\omega_0(t)^{n-1}}{(n-1)!} \right)\wedge dt.
		\end{aligned}
		\end{equation}where $(\omega_i(t)):=(\omega_0(t),\cdots,\omega_m(t)), t\in [0,s],$ is a smooth path from $(\omega_i(s))$ to $(\omega_i(0))$ such that $\varphi_i(t)$ is a ray of smooth function with $\omega_i(t)=\omega_i(0)+dd^c\varphi_i(t)$, and the topological constant $\hat{S}$ is given by
		\begin{align*}
			\hat{S}:=-n\frac{\left(K_X+\sum_{i=0}^m[\omega_i]\right)[\omega_0]^{n-1}}{[\omega_0]^n}.
		\end{align*}
	\end{definition}
	There are two immediate observations. Firstly, $\mathcal{M}_{\operatorname{ccscK}}$ includes the standard K-energy as part of it, namely\begin{equation}\label{splitMabuchi}
		\begin{aligned}
			&\mathcal{M}_{\operatorname{ccscK}}((\omega_i(s)),(\omega_i(0)))-\mathcal{M}_{\operatorname{cscK}}(\omega_0(s),\omega_0(0))\\& = -\sum_{i=1}^m\int_0^s\int_X\dot{\varphi_i}(t)\left(\frac{1}{V_i} \frac{\omega_i(t)^n}{n!}- \frac{1}{V_0}\frac{\omega_0(t)^n}{n!} \right)\wedge dt \\& + \frac{1}{V_0}\int_0^s\int_M \dot{\varphi_0}(t) \left(\overline{S}\frac{\omega_0(t)^n}{n!} + \sum_{i=0}^m\omega_i(t) \wedge \frac{\omega_0(t)^{n-1}}{(n-1)!} \right)\wedge dt,
		\end{aligned}
	\end{equation}where $\overline{S}=-n\frac{\sum_{i=0}^m[\omega_i].[\omega_0]^{n-1}}{[\omega_0]^n}$.
	
	Secondly, the defining 1-form for $\mathcal{M}_{\operatorname{ccscK}}$ can be written as the following pairing\begin{align*}
		-\frac{1}{V_0}\braket{\mu^{b}(I),\zeta_I(\cdot)},
	\end{align*}just like \cite[(3.103)]{alvarez2013coupled}. To explain this, let $I$ be the holomorphic structure on $E$ induced by $J$ and the holomorphic structure on $L_i$, and let $B_I:=\mathcal{K}_{J}([\omega_0])\times \{(\mathbb{S}^1)^m\text{-reductions of }E\}$ (see \cite[Section 3.1]{alvarez2013coupled}). Fix $\omega_i\in \mathcal{K}([\omega_i]),i=1,...,m$, take the embedding\begin{align*}
		&\prod_{i=0}^m\mathcal{K}([\omega_i])\hookrightarrow B_I,\\ &(\omega_0',\omega_1+\sqrt{-1}\partial\bar{\partial}\varphi_1,...,\omega_m+\sqrt{-1}\partial\bar{\partial}\varphi_m)\mapsto (\omega_0',e^{-2\pi\varphi_1}h_1,\cdots,e^{-2\pi\varphi_m}h_m),
	\end{align*}where $h_i$ is a Hermitian metric on $L_i$ with Chern curvature $-2\pi\sqrt{-1}\omega_i$, and $\varphi_i$ is smooth function that $\int_X\varphi_i\omega_i^n=0$. Given any  $(\omega_0',\omega_1+\sqrt{-1}\partial\bar{\partial}\varphi_1,...,\omega_m+\sqrt{-1}\partial\bar{\partial}\varphi_m)$, since $\omega_0'$ is compatible with $J$, $I$ is uniquely determined by a pair $(J,A)\in \mathcal{J}_{\omega_0'}\times \prod_{i=1}^m\mathcal{A}_{e^{-2\pi\varphi_i}h_i}$ (see \cite[Lemma 3.1]{alvarez2013coupled} and \cite{singer1959geometric}). So one can consider the picture given by Section \ref{subsec:2} with respect to the symplectic form $\omega_0'$ and the reduction $(e^{-2\pi\varphi_1}h_1,\cdots,e^{-2\pi\varphi_m}h_m)$. Let $b=(\omega_0',e^{-2\pi\varphi_1}h_1,\dots,e^{-2\pi\varphi_m}h_m)$, denote the corresponding moment map by $\mu^b$ and evaluate it at the point $I=(J,A)$ to obtain $\mu^b(I)$. Then one can easily show that the defining 1-form equals to $-\braket{\mu^{b}(I),\zeta_I(\cdot)}/V_0$, where $\zeta_I$ is defined by \cite[(3.122)]{alvarez2013coupled}. Therefore, it is reasonable that we can prove the exactness of $-\braket{\mu^{b}(I),\zeta_I(\cdot)}$ similar to \cite[Proposition 3.8]{alvarez2013coupled}. From the closedness of $-\braket{\mu^{b}(I),\zeta_I(\cdot)}$ 
	the well-definedness (independence of the choice of the path) of $\mathcal{M}_{\operatorname{ccscK}}$ follows, which is well-known as a standard fact. More generally, when the automorphism group of $X$ is non-discrete and the maximal torus acts
		as Hamiltonian isometries, the equation (\ref{Futaki}) is independent of choice of the K\"ahler forms, giving the coupled version of classical Futaki invariant. A standard fact is that the closed $1$-form {invariant under the automorphism action} as above, Futaki-type invariant, Mabuchi-type functional and moment map
		picture come as a package.
	\begin{proposition}\label{welldefinedMabuchi}
		$\mathcal{M}_{\operatorname{ccscK}}$ is well defined, i.e. the integration on right hand side of (\ref{splitMabuchi}) is independent of the choice of smooth path. For a Hamiltonian torus action, the equation (\ref{Futaki}) is independent of choice of the K\"ahler forms, giving the classical Futaki invariant (see \cite[Definition 10, Proposition 11]{datar2020coupled}).
	\end{proposition}
	\subsection{Multivariate energy functionals and their asymptotic slopes}
	\label{subsec:4}
	For energy functionals such as the Monge-Amp\`{e}re energy functional, Sj\"{o}str\"{o}m Dyrefelt introduced a powerful tool in his work \cite{sjostrom2018k}, namely the \textit{multivariate energy functional} or {\textit{analytic Deligne pairing}}. In the theory of classical cscK equation, it provides a remarkably refined means of relating the asymptotic slope of K-energy $\operatorname{M}^{\operatorname{cscK}}$ with $\operatorname{M}^{\operatorname{NA}}$, the non-Archimedean Mabuchi functional (a modiﬁcation of the Donaldson-Futaki invariant, introduced by \cite{boucksom2017uniform,boucksom2019}).
	
	\begin{definition}[\hspace*{-0.6ex}\text{\cite[Definition 2.1]{sjostrom2018k}}]\label{analyticDP}
		Let $\theta_0,\dots,\theta_n$ be closed $(1,1)$-forms on $X$.  
		Define a multivariate energy functional $\braket{\cdot,\dots,\cdot}_{(\theta_0,\ldots,\theta_n)}$  
		on the space $\operatorname{PSH}(X,\theta_0)\cap L^\infty_{\operatorname{loc}}\times \cdots\times \operatorname{PSH}(X,\theta_n)\cap L^\infty_{\operatorname{loc}}$ ($n+1$ times) by
		\begin{align}\label{delignepairing}
			\langle \varphi_0,\dots,\varphi_n\rangle_{(\theta_0,\dots,\theta_n)}
			&:= 
			\int_X \varphi_0\,(\theta_1 + dd^c\varphi_1)\wedge\dots\wedge(\theta_n + dd^c\varphi_n)\notag \\&+\int_X \varphi_1\,\theta_0\wedge(\theta_2 + dd^c\varphi_2)\wedge\dots\wedge(\theta_n + dd^c\varphi_n)\notag\\& +\int_X \varphi_2\,\theta_0\wedge\theta_1\wedge(\theta_3 + dd^c\varphi_3)\wedge\dots\wedge(\theta_n + dd^c\varphi_n)
			\notag\\ &+ \dots
			+ \int_X \varphi_n\,\theta_0\wedge\dots\wedge\theta_{n-1}, 
		\end{align}where $\bigwedge_{j=k}^n(\theta_j+dd^c\varphi_j)$ is a closed positive current in the sense of Bedford--Taylor \cite{BedfordTaylor76,BedfordTaylor82}.
	\end{definition}
	\begin{example}
		The Monge-Amp\`ere energy functional can be rewritten as\begin{align*}
			\operatorname{E}_\omega(\varphi):=\frac{1}{n+1}\sum_{j=0}^n\int_X\varphi(\omega+dd^c\varphi)^{n-j}\wedge\omega^j=\frac{1}{n+1}\braket{\varphi,\cdots,\varphi}_{(\omega,\cdots,\omega)},
		\end{align*}for any $\varphi\in\operatorname{PSH}(X,\omega)\cap L^\infty(X)$. It is well known that
		\begin{align*}
			\left.\frac{d}{dt}\right|_{t=0}\operatorname{E}_\omega(\varphi+t\psi)
			=\int_X\psi\,(\omega+dd^c\varphi)^n,
		\end{align*}
		Consequently, for any smooth path
		$\varphi(t),t\in[0,s]$ in $\mathcal{H}(\omega)$
		joining $0$ to $\varphi$, one has
		\begin{align}\label{MongeAmpere}
			\frac{1}{n+1}
			\braket{\varphi,\ldots,\varphi}_{(\omega,\ldots,\omega)}
			=\int_0^s\int_X
			\dot{\varphi}(t)\,
			(\omega+dd^c\varphi(t))^n\wedge dt,
		\end{align}
		where the right-hand side is independent of the choice of path.
		
		Also, for the contracted Monge-Amp\`ere energy functional, one has\begin{align*}\operatorname{E}_\omega^\chi(\varphi):=\frac{1}{n}\sum_{j=0}^{n-1}\int_X\varphi\chi\wedge(\omega+dd^c\varphi)^{n-j-1}\wedge\omega^j=\frac{1}{n}\braket{0,\varphi,\cdots,\varphi}_{(\chi,\omega,\cdots,\omega)},\end{align*} for any $\varphi\in \operatorname{PSH}(X,\omega)\cap L^\infty(X)$.
	\end{example}
	\noindent The utility of the multivariate energy functional is twofold. First, this functional enjoys the following properties: 
	\begin{proposition}[\hspace*{-0.6ex}\text{\cite[Proposition 2.3 \& 2.8]{sjostrom2018k}}]\mbox{}\label{duichen}\\
		$(\mathrm{i})$ $\braket{\cdot,\dots,\cdot}_{(\theta_0,\ldots,\theta_n)}$ is symmetric, i.e. for any $\sigma\in S_{n+1}$, $$\braket{\varphi_{\sigma(0)},\dots,\varphi_{\sigma(n)}}_{(\theta_{\sigma(0)},\dots,\theta_{\sigma(n)})}=\braket{\varphi_0,\dots,\varphi_n}_{(\theta_0,\dots,\theta_n)}.$$
		$(\mathrm{ii})$ Let $(\varphi_i(t))_{t\geq 0}$ be a smooth ray of smooth functions, $\Phi_i:X\times \mathbb{D}^*\to \mathbb{R},\Phi_i(x,e^{-t+\sqrt{-1}s})=\varphi_i^t(x)$. Then
		\begin{align}\label{derivativeofanalyticDelignepairing}
			\frac{d}{dt}\bigg|_{t=s}\braket{\varphi_0(t),\dots,\varphi_n(t)}_{(\theta_0,\dots,\theta_n)}=\sum_{i=0}^n\int_X\dot{\varphi_i}(s)\bigwedge_{j\neq i}(\theta_j+dd^c\varphi_j(s)),
		\end{align}and\begin{align}\label{ddcofanalyticDelignepairing}
			dd^c_\tau\braket{\varphi_0(t),\dots,\varphi_n(t)}_{(\theta_0,\dots,\theta_n)}=\int_X\bigwedge_{i=0}^n(\operatorname{pr}_1^*\theta_i+dd^c_{(x,\tau)}\Phi_i),
		\end{align}where $\tau=e^{-t+\sqrt{-1}s}$, $\operatorname{pr}_1:X\times \mathbb{D}^*\to X$ is projection.
	\end{proposition}\noindent Property (ii) readily implies the convexity of such functionals along suitable paths. As for (i), we actually have the following identity, which shows the essence of its symmetry (see also \cite[Appendix A.8, page 121]{BoucksomEriksson2021}).
	\begin{proposition}\label{BigSymmetryProp}
		Let $\omega_i,\ i=0,\cdots, n,$ be K\"{a}hler forms. Then for any $\varphi_i\in \operatorname{PSH}(\omega_i)\cap L^\infty(X),\ i=0,\cdots,n$, we have\begin{align}\label{BigSymmetry}
			\braket{\varphi_0,\cdots,\varphi_n}_{(\omega_0,\cdots,\omega_n)}=\frac{1}{n!}\sum_{\varnothing\neq I\subset \{0,\cdots,n\}}(-1)^{n+1-|I|}\operatorname{E}_{\sum_{i\in I}\omega_i}\left(\sum_{i\in I}\varphi_i\right).
		\end{align}
	\end{proposition}
	\begin{proof}
		Consider first the smooth case. Denote the right hand side by $[\varphi_0,\cdots,\varphi_n]_{(\omega_0,\cdots,\omega_n)}$, then we know that $[0,\cdots,0]_{(\omega_0,\cdots,\omega_n)}=0$. 
		We use {the following identity:}
			\begin{equation}\label{symmetry}
				\sum_{\begin{subarray}{c}
						\varepsilon_k \in \{0,1\}\\
						k = 1, \cdots, n
				\end{subarray}} (-1)^{n -\sum_k \varepsilon_k} (a + \varepsilon_1 b_1 + \cdots + \varepsilon_n b_n)^n = n!\ b_1 \cdots b_n.
		\end{equation}
		Put for convenience $a_i = \psi_i$ and $c_j = \omega_j + dd^c \varphi_j$. Then using (\ref{symmetry}) we obtain
		\begin{align*}
				&\frac{d}{dt}\bigg|_{t=0}[\varphi_0+t\psi_0,\cdots,\varphi_n+t\psi_n]_{(\omega_0,\cdots,\omega_n)}\\ =&\ \frac{1}{n!}\sum_{\varnothing\neq I\subset \{0,\cdots,n\}}(-1)^{n+1-|I|}\int_X\left(\sum_{i\in I}a_i\right)\left(\sum_{j\in I}c_j\right)^n\\
				=&\ \frac{1}{n!}\sum_{i\in I}a_i\sum_{J \subset \{0, \cdots, n\}\backslash \{i\}} (-1)^{n-|J|}(c_i + \sum_{j\in J} c_j)^n\\
				=&\ \frac{1}{n!}\sum_{i\in I}a_i\sum_{\begin{subarray}{c}
						\varepsilon_k \in \{0,1\}\\
						k = 0, \cdots, n\\
						k\ne i
				\end{subarray}}(-1)^{n-\sum_k \varepsilon_k }(c_i + \sum_{k\ne j} \varepsilon_k c_k)^n\\
				=& \sum_{i=0}^n a_i \bigwedge_{k\ne i} c_k\\
				=& \sum_{i=0}^n \psi_i \bigwedge_{k\ne i} (\omega_k + dd^c \varphi_k).
		\end{align*}
		By (\ref{derivativeofanalyticDelignepairing}), the identity holds for $\varphi_i$ smooth.
		
		Now for $\varphi_i\in \operatorname{PSH}(\omega_i)\cap L^\infty(X)$, take a decreasing sequence of smooth $\omega_i$-psh functions $\varphi_i^\varepsilon$ that descends to $\varphi_i$ (c.f. \cite[Proposition 8.16]{GZ17}). By the continuity of the Bedford--Taylor product under decreasing limits (c.f. \cite[Theorem 3.18]{GZ17}), we have\begin{align*}
			\varphi_i^\varepsilon\bigwedge_{j=0}^{i-1}\omega_j\wedge \bigwedge_{l=i+1}^n(\omega_l+dd^c\varphi_l^\varepsilon)\to \varphi_i\bigwedge_{j=0}^{i-1}\omega_j\wedge \bigwedge_{l=i+1}^n(\omega_l+dd^c\varphi_l),
		\end{align*}and\begin{align*}
			\left(\sum_{i\in I}\varphi_i^\varepsilon\right)\left(\sum_{i\in I}\omega_i\right)^j\wedge \left(\sum_{i\in I}\omega_i+dd^c\varphi_i^\varepsilon\right)^{n-j}\to \left(\sum_{i\in I}\varphi_i\right)\left(\sum_{i\in I}\omega_i\right)^j\wedge \left(\sum_{i\in I}\omega_i+dd^c\varphi_i\right)^{n-j}
		\end{align*}in weak sense of current. Thus\begin{align*}
			&\braket{\varphi_0,\cdots,\varphi_n}_{(\omega_0,\cdots,\omega_n)}=\lim_{\varepsilon\to 0}\braket{\varphi_0^\varepsilon,\cdots,\varphi_n^\varepsilon}_{(\omega_0,\cdots,\omega_n)}\\=&\ \lim_{\varepsilon\to 0}\frac{1}{n!}\sum_{\varnothing\neq I\subset \{0,\cdots,n\}}(-1)^{n+1-|I|}\operatorname{E}_{\sum_{i\in I}\omega_i}\left(\sum_{i\in I}\varphi_i^\varepsilon\right)\\=& \frac{1}{n!}\sum_{\varnothing\neq I\subset \{0,\cdots,n\}}(-1)^{n+1-|I|}\operatorname{E}_{\sum_{i\in I}\omega_i}\left(\sum_{i\in I}\varphi_i\right)
		\end{align*}as expected.
	\end{proof}
	
	The second utility of the multivariate energy functional is that, along $\mathcal{H}(\omega)$-rays compatible with a test configuration in the sense of Definition \ref{compatible}, its slope at infinity can be expressed as an intersection number.
	\begin{lemma}[\hspace*{-0.6ex}\text{\cite[Theorem 4.9]{sjostrom2018k}}]\label{slopeofAnalyticDelignePairing}
		Let $X$ be a compact K\"{a}hler manifold of dimension $n$, and $\theta_i,i=0,...,n$ closed $(1,1)$-forms on $X$, and let $[\theta_i]$ be the cohomology class of $\theta_i$ for $i=0,...,n$. Consider smooth cohomological test conﬁgurations $(\mathcal{X}_i,\mathcal{A}_i)$ for $(X,[\theta_i])$ dominating $X\times\mathbb{P}^1$. For each collection of smooth rays $(\varphi_i^t)_{t\geq 0}$ of smooth functions that $\mathcal{C}^{\infty}$-compatible with $(\mathcal{X}_i,\mathcal{A}_i),i=0,...,n$, we have\begin{align*}
			\lim_{t\to\infty}\frac{\braket{\varphi_0^t,\dots,\varphi_n^t}_{(\theta_0,\ldots,\theta_n)}}{t}=(\mathcal{A}_0\cdot\dots\cdot \mathcal{A}_n),
		\end{align*}by pulling back the cohomology classes $\mathcal{A}_i$ to a smooth test configuration $\mathcal{X}$ that dominates all $\mathcal{X}_i$ and $X\times\mathbb{P}^1$ (Such a test configuration exists by equivariant resolution of singularities, c.f. \cite[Section 3.1]{sjostrom2018k}).
	\end{lemma}

	We mention at the end of this section that (\ref{ddcofanalyticDelignepairing}) hold in the weak sense (see also the proof of \cite[Proposition 2.4.10]{dyrefelt2017k}).  In other words, we have the following Proposition.
	\begin{proposition}\label{ddcofanalyticDelignepairingweak}
		Let $\omega_i,i=0,\cdots,n$ be K\"{a}hler forms, and $\varphi_i,\varphi_i'\in \mathcal{H}(\omega_i)$. Let $\varphi_i(t)$ be the weak geodesic in $\mathcal{H}^{1,\bar{1}}(\omega_i)$ connecting $\varphi_i$ and $\varphi_i'$, then\\
		$(\mathrm{i})$ $f(t):=\braket{\varphi_0(t),\cdots,\varphi_n(t)}_{(\omega_0,\cdots,\omega_n)}$ is Lipschitz.\\
		$(\mathrm{ii})$ \begin{align}
			\frac{d}{dt}\bigg|_{t=s}\braket{\varphi_0(t),\dots,\varphi_n(t)}_{(\omega_0,\dots,\omega_n)}=\sum_{i=0}^n\int_X\dot{\varphi_i}(s)\bigwedge_{j\neq i}(\omega_j+dd^c\varphi_j(s)).
		\end{align}
		$(\mathrm{iii})$  Let $\Phi_i:X\times \mathbb{D}^*\to \mathbb{R},\Phi_i(x,e^{-t+\sqrt{-1}s})=\varphi_i^t(x)$. Then\begin{align*}
			dd^c_\tau\braket{\varphi_0(t),\dots,\varphi_n(t)}_{(\omega_0,\dots,\omega_n)}=\int_X\bigwedge_{i=0}^n(\operatorname{pr}_1^*\omega_i+dd^c_{(x,\tau)}\Phi_i),
		\end{align*}where $\tau=e^{-t+\sqrt{-1}s}$, $\operatorname{pr}_1:X\times \mathbb{D}^*\to X$ is projection. In particular $f(t)$ is convex (in the point-wise sense).
	\end{proposition}
	\begin{proof}
		(i) By the symmetry of analytic Deligne pairing, we have\begin{align*}
			&\braket{\varphi_0(s),\cdots,\varphi_n(s)}_{(\omega_0,\cdots,\omega_n)}-\braket{\varphi_0(t),\cdots,\varphi_n(t)}_{(\omega_0,\cdots,\omega_n)}\\=&\ \sum_{i=0}^n\left(\braket{\varphi_0(s),\cdots,\varphi_i(s),\varphi_{i+1}(t),\cdots\varphi_n(t)}_{(\omega_0,\cdots,\omega_n)}\right. \\&-\left.\braket{\varphi_0(s),\cdots,\varphi_{i-1}(s),\varphi_{i}(t),\cdots\varphi_n(t)}_{(\omega_0,\cdots,\omega_n)}\right)\\=&\ \sum_{i=0}^n\int_X(\varphi_i(s)-\varphi_i(t))\bigwedge_{j=0}^{i-1}(\omega_j+dd^c\varphi_j(s))\wedge \bigwedge_{l=i+1}^n(\omega_l+dd^c\varphi_l(t)).
		\end{align*}Note that weak geodesic is uniformly Lipschitz in $t$ (c.f. \cite[Section 2.2]{Berndtsson2015}) and the integration $\int_X\bigwedge_{j=0}^{i-1}(\omega_j+dd^c\varphi_j(s))\wedge \bigwedge_{l=i+1}^n(\omega_l+dd^c\varphi_l(t))$ is computed in cohomology, so $f(t)$ is also Lipschitz.
		
		(ii, iii) follow from (\ref{BigSymmetry}) and the well-known formula for the derivative and $dd^c_\tau$ of $\operatorname{E}$ in $L^\infty$ case (see for example \cite[(3.4) (3.5)]{BB17}). For (iii), we can also use the approximation method directly to show that\begin{align*}
			\int_{\mathbb{D}^*}\braket{\varphi_0(t),\dots,\varphi_n(t)}_{(\omega_0,\dots,\omega_n)}dd^cf=\int_{X\times \mathbb{D}^*}f\bigwedge_{i=0}^n(\operatorname{pr}_1^*\omega_i+dd^c_{(x,\tau)}\Phi_i)
		\end{align*}holds for any $f\in\mathcal{C}^\infty(\mathbb{D}^*)$ with compact support. Suppose that $\operatorname{Supp}(f)$ is contained in an open precompact annulus $A\subset\mathbb{D}^*$. Take also slightly larger nested annuli $A\subsetneq A'\subsetneq A''\subset \mathbb{D}^*$.
		
		We claim that there exists smooth $\mathbb{S}^1$-invariant $\operatorname{pr}_1^*\omega_i$-psh functions $\Phi_i^\varepsilon$ for $i=0,\cdots,n$ such that $\Phi_i^\varepsilon\downarrow \Phi_i$ on $X \times A'$.  To do this, take $\chi$ smooth strictly psh on ${A''}$. Then $\Omega_i:=\operatorname{pr}_1^*\omega_i+dd^c\chi$ is K\"{a}hler on $X\times {A''}$. By \cite[Theorem 2]{BlockiKolodziej2007}, there exists $\Psi_i^\varepsilon\in \operatorname{PSH}(X\times A',\operatorname{pr}_1^*\omega_i+\varepsilon\Omega_i)\cap \mathcal{C}^\infty(X\times A')$ such that $\Psi_i^\varepsilon\downarrow \Phi_i$. After taking $\mathbb{S}^1$-average to $\Psi_i^\varepsilon$ and adding a constant to $\chi$, we may assume that $\Psi_i^\varepsilon$ is $\mathbb{S}^1$-invariant and less than $\chi$. Then $\Phi_i^\varepsilon=\frac{\Psi_i^\varepsilon+\varepsilon\chi}{1+\varepsilon}$  is $\operatorname{pr}_1^*\omega_i$-psh and decreasing 
			as $\varepsilon \downarrow 0$ by the choice $\chi>\Phi_i^\varepsilon$.
		
		Now by the continuity of the Bedford--Taylor product under decreasing limits (c.f. \cite[Theorem 3.18]{GZ17}), we have\begin{align*}
			\bigwedge_{i=0}^n(\operatorname{pr}_1^*\omega_i+dd^c_{(x,\tau)}\Phi_i^\varepsilon)\to \bigwedge_{i=0}^n(\operatorname{pr}_1^*\omega_i+dd^c_{(x,\tau)}\Phi_i)
		\end{align*}on $X\times A'$, and\begin{align*}
			\varphi_i^\varepsilon(t)\bigwedge_{j=0}^{i-1}\omega_j\wedge \bigwedge_{l=i+1}^n(\omega_l+dd^c\varphi_l^\varepsilon(t))\to \varphi_i(t)\bigwedge_{j=0}^{i-1}\omega_j\wedge \bigwedge_{l=i+1}^n(\omega_l+dd^c\varphi_l(t))
		\end{align*}on $X$, where $\varphi_i^\varepsilon=\Phi_i^\varepsilon|_{X\times \{e^{-t}\}}$. Both limits hold in the weak sense of currents. Taking integration, we have\begin{align*}
			&\int_{\mathbb{D}^*}\braket{\varphi_0(t),\dots,\varphi_n(t)}_{(\omega_0,\dots,\omega_n)}dd^cf=\lim_{\varepsilon\to 0}\int_{\mathbb{D}^*}\braket{\varphi_0^\varepsilon(t),\dots,\varphi_n^\varepsilon(t)}_{(\omega_0,\dots,\omega_n)}dd^cf\\
			=&\ \lim_{\varepsilon\to 0}\int_{X\times \mathbb{D}^*}f\bigwedge_{i=0}^n(\operatorname{pr}_1^*\omega_i+dd^c_{(x,\tau)}\Phi_i^\varepsilon)=\int_{X\times \mathbb{D}^*}f\bigwedge_{i=0}^n(\operatorname{pr}_1^*\omega_i+dd^c_{(x,\tau)}\Phi_i),
		\end{align*}so the Proposition follows.
	\end{proof}
	
	\subsection{A Chen-Tian formula for the coupled K-energy}
	\label{subsec:5}
	Next we want to relate the coupled K-energy to the multivariate energy functional. By Proposition \ref{welldefinedMabuchi}, we can compute $\mathcal{M}_{\operatorname{ccscK}}((\omega_{i}+dd^c\varphi_i),(\omega_{i}))$ by choosing a special line $\gamma(t)$\begin{align*}
		&(\omega_i(t)):=(\omega_0+tdd^c\varphi_0,\omega_1,\cdots,\omega_m), t\in [0,1],\ \text{connected behind}\\
		&(\omega_i(t')):=(\omega_0+dd^c\varphi_0,\omega_1+t'dd^c\varphi_1,\cdots,\omega_m+t'dd^c\varphi_m), t'\in [0,1].
	\end{align*}We only need to compute the value of\begin{align*}
		&-\sum_{i=1}^m\int_0^s\int_X\dot{\varphi_i}(t)\left(\frac{1}{V_i} \frac{\omega_i(t)^n}{n!}- \frac{1}{V_0}\frac{\omega_0(t)^n}{n!} \right)\wedge dt \\&+ \frac{1}{V_0}\int_0^s\int_M \dot{\varphi_0}(t) \left(\overline{S}\frac{\omega_0(t)^n}{n!} + \sum_{i=0}^m\omega_i(t) \wedge \frac{\omega_0(t)^{n-1}}{(n-1)!} \right)\wedge dt
	\end{align*}because of (\ref{splitMabuchi}). That is\begin{align*}
		&-\sum_{i=1}^m\int_0^1\int_X\varphi_i\left(\frac{1}{V_i} \frac{(\omega_i+tdd^c\varphi_i)^n}{n!}-\frac{1}{V_0} \frac{(\omega_0+dd^c\varphi_0)^n}{n!} \right)\wedge dt \\& + \frac{1}{V_0}\int_0^1\int_M \varphi_0 \left(\left(\overline{S}+n\right)\frac{(\omega_0+tdd^c\varphi_0)^n}{n!} + \sum_{i=1}^m\omega_i \wedge \frac{(\omega_0+tdd^c\varphi_0)^{n-1}}{(n-1)!} \right)\wedge dt.
	\end{align*}
	The component $\int_0^1\int_X\varphi_i\frac{(\omega_i+tdd^c\varphi_i)^n}{n!}\wedge dt=\frac{1}{(n+1)!}\braket{\varphi_i,\cdots,\varphi_i}_{(\omega_i,\cdots,\omega_i)}$ by (\ref{MongeAmpere}). 
		For the rest, note that\begin{align}\label{CT1}
			&\braket{\varphi_i,\varphi_0,\cdots,\varphi_0}_{(\omega_i,\omega_0,\cdots,\omega_0)}=\int_X\varphi_i(\omega_0+dd^c\varphi_0)^n+\sum_{q=0}^{n-1}\int_X\varphi_i\omega_i\wedge\omega_0^q\wedge (\omega_0+dd^c\varphi_0)^{n-1-q}
		\end{align}
		and
		\begin{align}
			&\sum_{q=0}^{n-1}\int_X\varphi_i\omega_i\wedge\omega_0^q\wedge (\omega_0+dd^c\varphi_0)^{n-1-q}
			\nonumber
			\\=&\ \sum_{q=0}^{n-1}\sum_{k=0}^q\int_XC_{n-1-q}^k\varphi_i\omega_i\wedge\omega_0^{q+k}\wedge (dd^c\varphi_0)^{n-1-q-k}
			\nonumber
			\\=&\ \sum_{q=0}^{n-1}\sum_{k=0}^q\int_X\left(C_{n-1-q+1}^{k}-C_{n-1-q}^{k-1}\right)\varphi_i\omega_i\wedge\omega_0^{q+k}\wedge (dd^c\varphi_0)^{n-1-q-k}
			\nonumber
			\\=&\ \sum_{j=0}^{n-1}\sum_{k=0}^j\int_X\left(C_{n-j+k}^{k}-C_{n-j+k-1}^{k-1}\right)\varphi_i\omega_i\wedge\omega_0^{j}\wedge (dd^c\varphi_0)^{n-1-j}\nonumber
			\\=&\ \sum_{j=0}^{n-1}\int_X\left(C_{n-j+j}^{j}-C_{n-j+0-1}^{0-1}\right)\varphi_i\omega_i\wedge\omega_0^{j}\wedge (dd^c\varphi_0)^{n-1-j}\nonumber
			\\=& \ \sum_{j=0}^{n-1}\int_XC_{n}^{j}\varphi_i\omega_i\wedge\omega_0^{j}\wedge (dd^c\varphi_0)^{n-1-j}.\label{CT2}
		\end{align}
		At the same time, we have\begin{align}
			&\int_0^1\int_M\varphi_0 \omega_i \wedge (\omega_0+tdd^c\varphi_0)^{n-1}\wedge dt=\sum_{j=0}^{n-1}\int_0^1\int_M C_{n-1}^j\varphi_0\omega_i \wedge \omega_0^j\wedge t^{n-1-j}(dd^c\varphi_0)^{n-j-1}\wedge dt\nonumber
			\\ &=\sum_{j=0}^{n-1}\int_M \frac{1}{n-j}C_{n-1}^j\varphi_0\omega_i \wedge \omega_0^j\wedge (dd^c\varphi_0)^{n-j-1} dt=\frac{1}{n}\sum_{j=0}^{n-1}\int_M C_{n}^j\varphi_0\omega_i \wedge \omega_0^j\wedge (dd^c\varphi_0)^{n-j-1}.\label{CT3}
		\end{align}Thus\begin{align*}
			\int_0^1\int_M\varphi_0 \omega_i \wedge (\omega_0+tdd^c\varphi_0)^{n-1}\wedge dt+\frac{1}{n}\int_X\varphi_i(\omega_0+dd^c\varphi_0)^n=\frac{1}{n}\braket{\varphi_i,\varphi_0,\cdots,\varphi_0}_{(\omega_i,\omega_0,\cdots,\omega_0)}
		\end{align*}and then
		\begin{theorem}
			\begin{equation}\label{our}
				\begin{aligned}
					&\mathcal{M}_{\operatorname{ccscK}}((\omega_{i}+dd^c\varphi_i),(\omega_{i}))= \mathcal{M}_{\operatorname{cscK}}(\omega_{0}+dd^c\varphi_0,\omega_{0})+\frac{\overline{S}+n}{V_0(n+1)!}\braket{\varphi_0,\cdots,\varphi_0}_{(\omega_0,\cdots,\omega_0)}\\ &\quad -\sum_{i=1}^m\frac{1}{V_i(n+1)!}\braket{\varphi_{i},\cdots,\varphi_{i}}_{(\omega_i,\cdots,\omega_i)}+\frac{1}{V_0n!}\sum_{i=1}^m\braket{\varphi_i,\varphi_0,\cdots,\varphi_0}_{(\omega_i,\omega_0,\cdots,\omega_0)}.
				\end{aligned}
			\end{equation}
		\end{theorem}
		\begin{corollary}\mbox{}\\
			$(\mathrm{i})$ The coupled K-energy functional extends to the space $\prod_{i=0}^m\mathcal{H}^{1,\bar{1}}(\omega_i)$.\\
			$(\mathrm{ii})$ Let $\varphi_i,\varphi_i'\in \mathcal{H}(\omega_i)$ and $\varphi_i(t)$ be the weak geodesic in $\mathcal{H}^{1,\bar{1}}(\omega_i)$ connecting them, then $\mathcal{M}_{\operatorname{ccscK}}((\omega_i+dd^c\varphi_i(t)),(\omega_i))$ is continuous and convex in $t$. In particular\begin{equation}\label{minimize1}
				\begin{aligned}
					&\lim_{t\to 0}\frac{\mathcal{M}_{\operatorname{ccscK}}((\omega_i+dd^c\varphi_i(t)),(\omega_i))-\mathcal{M}_{\operatorname{ccscK}}((\omega_i+dd^c\varphi_i(0)),(\omega_i))}{t}\\ &\geq \frac{1}{n!}\left[\frac{1}{V_0}\int_X\dot{\varphi_0}(0)\left((\hat{S}-S_{\omega_0(0)})\omega_0(0)^n+\sum_{i=0}^mn\omega_i(0)\wedge{\omega_0}(0)^{n-1}\right)\right. \\&\ \left. +\sum_{i=1}^m\dot{\varphi_i}(0)\left(\frac{\omega_0(0)^n}{V_0}-\frac{\omega_i(0)^n}{V_i}\right)\right]
				\end{aligned}
			\end{equation}where $\omega_i(t)=\omega_i+dd^c\varphi_i(t).$
		\end{corollary}
		\begin{proof}
			By the Chen-Tian formula \cite[(1)]{chen2000lower}, $\mathcal{M}_{\operatorname{cscK}}$ extends to $\mathcal{H}^{1,\bar{1}}$. Since the multivariate energy functional is well-defined on $\operatorname{PSH}\cap L^\infty_{\operatorname{loc}}$ in the sense of Bedford--Taylor, (i) follows directly.
			
			{By Propsition \ref{ddcofanalyticDelignepairingweak} (ii,iii), we know that $\braket{\varphi_{i}(t),\cdots,\varphi_{i}(t)}_{(\omega_i,\cdots,\omega_i)}$ is affine along weak geodesic $\varphi_i(t)$, and\begin{align*}
					dd^c_\tau\braket{\varphi_i(t),\varphi_0(t),\cdots,\varphi_0(t)}_{(\omega_i,\omega_0,\cdots,\omega_0)}=\int_X(\operatorname{pr}_1^*\omega_i+dd^c\Phi_i)\wedge (\operatorname{pr}_1^*\omega_0+dd^c\Phi_0)^n\geq 0,
				\end{align*}which means $\braket{\varphi_i(t),\varphi_0(t),\cdots,\varphi_0(t)}_{(\omega_i,\omega_0,\cdots,\omega_0)}$ is convex. On the other hand, $\mathcal{M}_{\operatorname{cscK}}(\omega_0+dd^c\varphi_0(t),\omega_0)$ is convex by \cite[Theorem 1.1]{BB17}. Thus, all of the above shows that $\mathcal{M}_{\operatorname{ccscK}}((\omega_i+dd^c\varphi_i(t)),(\omega_i))$ is convex. (\ref{minimize1}) is then implied by Proposition \ref{ddcofanalyticDelignepairingweak} (ii) and \cite[Lemma 3.5]{BB17}.}
		\end{proof}
		We are now ready for
		\begin{proof}[Proof of Theorem \ref{our2}]
			We use formula (\ref{our}) to compute the slope of each term separately. By \cite[Theorem C]{sjostrom2018k}, we have\begin{align*}
				\lim_{t\to\infty}\frac{\mathcal{M}_{\operatorname{cscK}}\left(\omega_0+dd^c\varphi_0(t),\omega_0\right)}{t}&=\operatorname{M}^{\operatorname{NA}}(\mathcal{X}_0, \mathcal{A}_0).
			\end{align*}As for the Deligne pairing terms, by \cite[Lemmas 4.4 and 4.8]{sjostrom2018k}, we can replace each ray $(\varphi_i(t))_{t\geq 0}$ by a smooth $\mathcal{C}^{\infty}$-compatible ray without changing the slope. Then the Theorem follows by Lemma \ref{slopeofAnalyticDelignePairing}.
		\end{proof}
		\begin{proof}[Proof of Theorem \ref{our2alg}]
			Note that for polarized case, the definition of analytic Deligne pairing identifies with the Deligne pairing defined in \cite[Section 1.3]{boucksom2019} by \cite[(1.5)]{boucksom2019}. The proof is thus completed by applying \cite[Lemma 3.9]{boucksom2019} to (\ref{our}) and invoking \cite[Proof of Theorem 3.6]{boucksom2019}.
		\end{proof}
		\begin{proof}[Proof of Theorem \ref{our6}]
			Suppose that $([\omega_i])$ admits a coupled K\"{a}hler-Einstein metric $(\omega_i')$. For any $(\omega_i'')$, take weak geodesics $\varphi_i(t)$ connecting the potential of  $(\omega_i')$ and  $(\omega_i'')$. Let $\omega_i(t):=\omega_i+dd^c \varphi_{i}(t)$. Then by (\ref{minimize1}), we know that\begin{align*}
				&\lim_{t\to 0}\frac{\mathcal{M}_{\operatorname{ccscK}}((\omega_i(t)),(\omega_i))-\mathcal{M}_{\operatorname{ccscK}}((\omega_i'),(\omega_i))}{t}\geq 0.
			\end{align*}By the convexity of $\mathcal{M}_{\operatorname{ccscK}}$ along $(\varphi_i(t))$, we have $\mathcal{M}_{\operatorname{ccscK}}((\omega_i''),(\omega_i))\geq \mathcal{M}_{\operatorname{ccscK}}((\omega_i'),(\omega_i))$ for any $\omega_i''\in [\omega_i]$. Then (i) and (ii) follow from Theorem \ref{our2} and \ref{our2alg}.
		\end{proof}
		Finially, by Theorem \ref{our6} and the inequality $\operatorname{M}^{\operatorname{NA}}\leq\operatorname{DF}$, Theorem \ref{coupledKsemistable} follows from the following Proposition, which extends \cite[Propositions 3.12, 3.17]{sjostrom2018k} (see also \cite[Proposition 2.23]{DR}).
		\begin{proposition}
			Let $(X,([\omega_i])_{i=0}^m)$ be a compact K\"ahler manifold with a tuple of K\"ahler classes, then\\
			$(\mathrm{i})$ $(X,([\omega_i])_{i=0}^m)$ is coupled K-semistable if and only if for any tuple of smooth (compactified) relatively K\"{a}hler cohomological test conﬁgurations $(\mathcal{X}_i,\mathcal{A}_i)_{i=0}^m$ that dominates $X\times\mathbb{P}^1$, we have $\operatorname{DF}_c((\mathcal{X}_i,\mathcal{A}_i)_{i=0}^m)\geq 0$. In fact, we can further assume that their central fiber $\mathcal{X}_i(0)$ are reduced.\\
			$(\mathrm{ii})$ If further $[\omega_i]=c_1(L_i)$ for ample line bundles $L_i$, then $(X,(L_i)_{i=0}^m)$ is coupled algebraically K-semistable if and only if $(X,(c_1(L_i))_{i=0}^m)$ is coupled K-semistable.
		\end{proposition}
		\begin{proof}
			(i) We begin by verifying the first criterion. Let $(\mathcal{X}_i,\mathcal{A}_i)_{i=0}^m$ be a tuple of (compactified) relatively K\"{a}hler cohomological test conﬁgurations of $(X,\alpha_i)_{i=0}^m$. Following the proof of \cite[Proposition 3.12]{sjostrom2018k}, for every sufficiently small $\delta>0$, there exists a tuple of smooth (compactified) relatively K\"ahler cohomological test configurations $(\mathcal{X}_i',\mathcal{A}_{i,\delta})_{i=0}^m$, where the total spaces $\mathcal{X}_i'$ are independent of $\delta$ and each simultaneously dominates $X\times\mathbb{P}^1$ and $\mathcal{X}_i$, and where writing $\mu_i:\mathcal{X}_i'\to\mathcal{X}_i$ for the corresponding $\mathbb{C}^*$-equivariant bimeromorphic morphisms, we have $\mathcal{A}_{i,\delta}\to\mu_i^*\mathcal{A}_i$ in $H_{\mathrm{BC}}^{1,1}(\mathcal{X}_i',\mathbb{R})$ as $\delta\to 0$.
%
			
		Now assume $\operatorname{DF}_c((\mathcal{X}_i',\mathcal{A}_{i,\delta})_{i=0}^m)\geq 0$ for $\delta$ small enough. Letting $\delta\to 0$ and noting that $	\operatorname{DF}_c((\mathcal{X}_i',\mathcal{A}_{i,\delta})_{i=0}^m)$ is a polynomial in the cohomology classes, we obtain\begin{align*}
				\operatorname{DF}_c((\mathcal{X}_i,\mathcal{A}_i)_{i=0}^m)=\operatorname{DF}_c(\mathcal{X}_i',\mu_i^*\mathcal{A}_i)_{i=0}^m)=\lim_{\delta\to 0}\operatorname{DF}_c((\mathcal{X}_i',\mathcal{A}_{i,\delta})_{i=0}^m)\geq 0.
			\end{align*}This proves the `if'-part of the first criterion. The other direction holds by deﬁnition.
			
		We next show how to reduce to test conﬁgurations with reduced central ﬁbre. The proof is from \cite[Proposition 2.23]{DR}. Let $(\mathcal{X}_i,\mathcal{A}_i)_{i=0}^m$ be a tuple of (compactified) relatively K\"ahler cohomological test configurations for $(X,\alpha_i)_{i=0}^m$. Since $\mathcal{A}_i$ are relatively K\"ahler, there exists K\"ahler class $\beta$ on $\mathbb{P}^1$ such that $\mathcal{A}_i+\pi^*\beta$ are all K\"ahler. For each $\mathcal{A}_i+\pi^*\beta$, choose an $\mathbb{S}^1$-invariant repesentative $\Omega_i=\Theta_i+\pi^*\theta_i$, where $\Theta_i\in \mathcal{A}_i$ and $\theta_i\in \beta$ are both $\mathbb{S}^1$-invariant.
		
		Consider the map $u:\mathbb{P}^1\to\mathbb{P}^1$ given by $z\mapsto z^d$ and let $\mathcal{X}_i'$ be the normalisation of the pullback of $\mathcal{X}_i$ along $u$. We denote the induced finite map $\mathcal{X}_i'\to\mathcal{X}_i$ by $u_i$. Setting $\Omega_i':=u_i^*\Omega_i$, we have that, for $d$ sufficiently large and divisible,  $\mathcal{X}_i'$ has reduced central fiber and $\Omega_i$ is K\"ahler. Furthermore, we may choose a common $d$ for which this holds for all $i$.
		
		By equivariant resolution of singularities, there is a test configuration (without a cohomology class) $\mathcal{Y}$ which is smooth and dominates all $\mathcal{X}_i$ and $X\times\mathbb{P}^1$ via $\mathbb{C}^*$-equivariant bimeromorphic morphisms $\rho_i:\mathcal{Y}\to\mathcal{X}_i$ and $\rho:\mathcal{Y}\to X\times\mathbb{P}^1$, respectively (see \cite[Section 4]{sjostrom2018k}). Now let $\mathcal{Y}'$ be the normalisation of the pullback of $\mathcal{Y}$ along $u$, and (still) denote the induced ﬁnite map by $u$. Then we have diagrams$$
		\begin{tikzcd}
			\mathcal{Y} \arrow[rd, "\rho_i"] \arrow[d, "\rho"'] &             & \mathcal{Y}' \arrow[rd, "\rho_i'"] \arrow[d, "\rho'"'] &              \\
			X\times\mathbb{P}^1 \arrow[r, dashed]               & \mathcal{X}_i & X\times\mathbb{P}^1 \arrow[r, dashed]                  & \mathcal{X}_i'
		\end{tikzcd}
		$$and for each $i$, above two diagrams are connected and commute by maps (still denotes by) $u:X\times\mathbb{P}^1\to X\times\mathbb{P}^1$, $u_i:\mathcal{X}_i'\to \mathcal{X}_i$ and $u:\mathcal{Y}'\to \mathcal{Y}$.
		
		Now let $p:\mathcal{Z}'\to \mathcal{Y}'$ be a resolution. By \cite[Proposition 2.23, Lemma 2.24]{DR}, we have\begin{align*}
			\operatorname{DF}(\mathcal{Z}',p^*(\rho_0')^*u_0^*[\Omega_0])\leq d\operatorname{DF}(\mathcal{Y},\rho_0^*[\Omega_0]).
		\end{align*}On the other hand, we have\begin{small}
		\begin{align*}
		&\frac{S\left((p^*(\rho_0')^*u_0^*[\Omega_0])^{n+1}\right)_{\mathcal{Z}'}}{V_0(n+1)!}-\sum_{i=1}^m\frac{\left((p^*(\rho_i')^*u_i^*[\Omega_i])^{n+1}\right)_{\mathcal{Z}'}}{V_i(n+1)!}+\sum_{i=1}^m\frac{\left((p^*(\rho_i')^*u_i^*[\Omega_i])\cdot(p^*(\rho_0')^*u_0^*[\Omega_0])^n\right)_{\mathcal{Z}'}}{V_0n!}\\=&\ \frac{S\left((p^*u^*\rho_0^*[\Omega_0])^{n+1}\right)_{\mathcal{Z}'}}{V_0(n+1)!}-\sum_{i=1}^m\frac{\left((p^*u^*\rho_i^*[\Omega_i])^{n+1}\right)_{\mathcal{Z}'}}{V_i(n+1)!}+\sum_{i=1}^m\frac{\left((p^*u^*\rho_i^*[\Omega_i])\cdot(p^*u^*\rho_0^*[\Omega_0])^n\right)_{\mathcal{Z}'}}{V_0n!}\\=&\ d\left(\frac{S\left((\rho_0^*[\Omega_0])^{n+1}\right)_{\mathcal{Y}}}{V_0(n+1)!}-\sum_{i=1}^m\frac{\left((\rho_i^*[\Omega_i])^{n+1}\right)_{\mathcal{Y}}}{V_i(n+1)!}+\sum_{i=1}^m\frac{\left((\rho_i^*[\Omega_i])\cdot(\rho_0^*[\Omega_0])^n\right)_{\mathcal{Y}}}{V_0n!}\right).
		\end{align*}
		\end{small}Thus by next Lemma, we have\begin{align*}
			\operatorname{DF}_c((\mathcal{X}_i',u_i^*\mathcal{A}_i)_{i=0}^m)&=	\operatorname{DF}_c((\mathcal{X}_i',u_i^*[\Omega_i])_{i=0}^m):=\operatorname{DF}_c((\mathcal{Z}',p^*(\rho_i')^*u_i^*[\Omega_i])_{i=0}^m)\\&\leq d\operatorname{DF}_c((\mathcal{Y},\rho_i^*[\Omega_i])_{i=0}^m)=: d\operatorname{DF}_c((\mathcal{X}_i,[\Omega_i])_{i=0}^m)= d\operatorname{DF}_c((\mathcal{X}_i,\mathcal{A}_i)_{i=0}^m)
		\end{align*}which completes the proof.

			(ii) The `if' part follows from the projection formula and the fact that any ample test conﬁguration $(\mathcal{X},\mathcal{L})$ for $(X,L)$ corresponds to a relative K\"ahler cohomological test configuration $(\mathcal{X},c_1(\mathcal{L}))$ for $(X,c_1(L))$. 
			
			As for the `only if' direction. Suppose that $(X,(L_i)_{i=0}^m)$ is algebraically K-semistable and let $(\mathcal{X}_i,\mathcal{A}_i)_{i=0}^m$ be a tuple of (compactified) cohomological test conﬁguration for $(X,c_1(L_i))_{i=0}^m$. By (i) we can further assume that $(\mathcal{X}_i,\mathcal{A}_i)$ is smooth, relatively K\"ahler and dominating $X\times\mathbb{P}^1$ with a $\mathbb{C}^*$-equivariant bimeromorphic morphism $\mu_i$. Again following the proof of \cite[Proposition 3.14]{sjostrom2018k}, for any $i$, there exists a sequence of smooth (compactified) ample cohomological test configurations $(\mathcal{X}_i,\mathcal{L}_{i,k}),k\in \mathbb{N}$ for $(X,L_i)$ such that $\mathcal{L}_{i,k}$ is relatively ample $\mathbb{Q}$-line bundle over $\mathcal{X}_i$ and $c_1(\mathcal{L}_{i,k})\to \mathcal{A}_i$. Suppose now $(X,(L_i)_{i=0}^m)$ is coupled algebraically K-semistable, then $\operatorname{DF}_c((\mathcal{X}_i,\mathcal{L}_{i,k})_{i=0}^m)\geq 0$. Note that $	\operatorname{DF}_c((\mathcal{X}_i,\mathcal{L}_{i,k})_{i=0}^m))$ is a polynomial in the cohomology classes, we have\begin{align*}
				\operatorname{DF}_c((\mathcal{X}_i,\mathcal{A}_{i})_{i=0}^m)=\lim_{k\to \infty}\operatorname{DF}_c((\mathcal{X}_i,\mathcal{L}_{i,k})_{i=0}^m)\geq 0, 
			\end{align*}i.e. $(X,(c_1(L_i))_{i=0}^m)$ is coupled K-semistable.
		\end{proof}
	\begin{lemma}
		$\operatorname{DF}_c((\mathcal{X}_i,\mathcal{A}_i)_{i=0}^m)=\operatorname{DF}_c((\mathcal{X}_i,\mathcal{A}_i+c_i\pi^*\beta)_{i=0}^m),$ where $\beta\in H^{1,1}(\mathbb{P}^1,\mathbb{R})$ such that $[\mathbb{P}^1]\frown\beta=1$.
	\end{lemma}
	\begin{proof}
		After pulling back to a resolution, we may assume that $\mathcal{X}_i=\mathcal{X}$ ane $\mathcal{X}$ is smooth (see \cite[Section 4]{sjostrom2018k}). Let $\mathcal{B}_i=\mathcal{A}_i+c_i\pi^*\beta$, then\begin{align*}
			&\frac{S}{V_0(n+1)!}(\mathcal{A}_0^{n+1})_{\mathcal{X}}-\sum_{i=1}^m\frac{1}{V_i(n+1)!}(\mathcal{A}_i^{n+1})_{\mathcal{X}}+\frac{1}{V_0n!}\sum_{i=1}^m(\mathcal{A}_i\cdot\mathcal{A}_0^n)_{\mathcal{X}}\\
			={}&\frac{S}{V_0(n+1)!}\left((\mathcal{B}_0-c_0\pi^*\beta)^{n+1}\right)_{\mathcal{X}}-\sum_{i=1}^m\frac{1}{V_i(n+1)!}\left((\mathcal{B}_i-c_i\pi^*\beta)^{n+1}\right)_{\mathcal{X}}\\
			&+\frac{1}{V_0n!}\sum_{i=1}^m\left((\mathcal{B}_i-c_i\pi^*\beta)\cdot(\mathcal{B}_0-c_0\pi^*\beta)^n\right)_{\mathcal{X}}\\
			={}&\frac{S}{V_0(n+1)!}(\mathcal{B}_0^{n+1})_{\mathcal{X}}-\frac{Sc_0}{n!}-\sum_{i=1}^m\frac{1}{V_i(n+1)!}(\mathcal{B}_i^{n+1})_{\mathcal{X}}+\frac{1}{n!}\sum_{i=1}^m c_i\\
			&+\frac{1}{V_0n!}\sum_{i=1}^m(\mathcal{B}_i\cdot\mathcal{B}_0^n)_{\mathcal{X}}-\frac{1}{n!}\sum_{i=1}^m c_i-\frac{c_0}{V_0(n-1)!}\sum_{i=1}^m(\alpha_i\cdot\alpha_0^{n-1})_X\\
			={}&\frac{S}{V_0(n+1)!}(\mathcal{B}_0^{n+1})_{\mathcal{X}}-\sum_{i=1}^m\frac{1}{V_i(n+1)!}(\mathcal{B}_i^{n+1})_{\mathcal{X}}+\frac{1}{V_0n!}\sum_{i=1}^m(\mathcal{B}_i\cdot\mathcal{B}_0^n)_{\mathcal{X}}.
		\end{align*}The proof is then follows by \cite[Lemma 2.24]{DR}.
	\end{proof}
		\boldmath
		\subsection{$\mathcal{E}^1$-extension of the coupled K-energy}
		\unboldmath
		\begin{lemma}\label{extensionLemma1}
			Given any K\"{a}hler forms $\omega,\omega'$ such that $\omega<\omega'$, the map $\operatorname{PSH}(X,\omega)\hookrightarrow \operatorname{PSH}(X,\omega'),\varphi\mapsto\varphi$ can be restricted to $\mathcal{E}^1(X,\omega)\hookrightarrow\mathcal{E}^1(X,\omega')$ and is $d_1$-continuous. Moreover, it maps $d_{1,\omega}$-bounded subsets to $d_{1,\omega'}$-bounded subsets.
		\end{lemma}
		\begin{proof}
			By \cite[Section 3.3.1]{GuedjTrusiani2023}, the inclusion  $\operatorname{PSH}(X,\omega)\hookrightarrow \operatorname{PSH}(X,\omega'),\varphi\mapsto\varphi$ can be restricted to $\mathcal{E}^1(X,\omega)\hookrightarrow\mathcal{E}^1(X,\omega')$, so we only need to check the $d_1$-continuity. Let $d_{1,\omega}$ and $d_{1,\omega'}$ be the $d_1$-metric on $\mathcal{E}^1(X,\omega)$ and $\mathcal{E}^1(X,\omega')$ respectively. We will show that for any $u,v\in \mathcal{E}^1(X,\omega)$ with $d_{1,\omega}(0,u)\leq C,d_{1,\omega}(0,u')\leq C$, there exists a continuous function $f_C$ such that $f_C(0)=0$ and\begin{align*}
				d_{1,\omega'}(u,v)\leq f_C(d_{1,\omega}(u,v)).
			\end{align*}We may equivalently work with $\mathbf{I}_{1,\omega}(u,v)$ and $\mathbf{I}_{1,\omega'}(u,v)$ in view of Theorem \ref{Theorem3ofDarvas2}.
			
			Take $A=2^N>1$ such that $\omega'<A\omega$. Then we have $\omega+dd^c(u/A)>\frac{\omega'+dd^cu}{A}$ and thus\begin{align*}
				\mathbf{I}_{1,\omega'}(u,v)=\int_X|u-v|(\omega_u^{\prime n}+\omega_v^{\prime n})\leq A^n\int_X|u-v|(\omega_{(u/A)}^{n}+\omega_{(v/A)}^{n}) = A^{n-1}\mathbf{I}_{1,\omega}(u/A,v/A).
			\end{align*}Note that by \cite[Lemma 5.3]{Darvas2}, $d_{1,\omega}(0,u/A)\leq C^{\prime N}d_{1,\omega}(0,u), d_{1,\omega}(0,v/A)\leq C^{\prime N}d_{1,\omega}(0,v)$ for some constant $C'$. On the other hand, by \cite[Corollary 5.8]{Darvas2} and its proof, for each $R>0$ there exists
			a continuous function $f_R: \mathbb{R}\to\mathbb{R}$ with $f_R(0)=0$ such that
			\begin{align*}
			\int_X|u-v|\omega_h^n\leq f_R(d_{1,\omega}(u,v))
			\end{align*}for any $u,v,h\in\mathcal{E}^1(X,\omega)\cap\{w:d_{1,\omega}(0,w)\leq R\}.$ Take $R=\max\{C,C^{\prime N}C\}$, we have\begin{align*}
				\int_X|u-v|(\omega_{(u/A)}^{n}+\omega_{(v/A)}^{n})\leq 2f_R(d_1(u,v)).
			\end{align*}Thus by Theorem \ref{Theorem3ofDarvas2}, there exists $C>0$ such that\begin{align*}
				C^{-1}A^{-n+1}d_{1,\omega'}(u,v)\leq A^{-n+1}\mathbf{I}_{1,\omega'}(u,v)\leq \mathbf{I}_{1,\omega}(u/A,v/A)\leq 2f_R(d_{1,\omega}(u,v))
			\end{align*}and the Lemma follows.
		\end{proof}
		\begin{lemma}\label{extensionLemma2}
			Given any K\"{a}hler forms $\omega,\omega'$, we have\\
			$(\mathrm{i})$ $\mathcal{E}^1(X,\omega)\subset L^1((\omega'+dd^cv)^n)$ for any $v\in \mathcal{E}^1(X,\omega')$;\\
			$(\mathrm{ii})$ The map\begin{align*}
				\mathcal{E}^1(X,\omega)\times \mathcal{E}^1(X,\omega')\to\mathbb{R},(u,v)\mapsto\int_Xu\omega_v^{\prime n}
			\end{align*}is $d_{1,\omega}\times d_{1,\omega'}$-continuous and bounded on any $d_{1,\omega}\times d_{1,\omega'}$-bounded subset.
		\end{lemma}
		\begin{proof}
			(i) is followed by \cite[Proposition 4.1]{Nezza2016} and \cite[Lemma 4.4]{BBGZ}. For (ii), one only needs to prove the case that $\omega=\omega'$ because\begin{align*}
				\mathcal{E}^1(X,\omega)\times \mathcal{E}^1(X,\omega')\to\mathbb{R},(u,v)\mapsto\int_Xu\omega_v^{\prime n}
			\end{align*}can be decomposed into\begin{align*}
				\mathcal{E}^1(X,\omega)\times \mathcal{E}^1(X,\omega')\to \mathcal{E}^1(X,\omega')\times \mathcal{E}^1(X,\omega'),(u,v)\mapsto (u/A,v)
			\end{align*}for $A:A\omega'>\omega$ and\begin{align*}
				\mathcal{E}^1(X,\omega')\times \mathcal{E}^1(X,\omega')\to\mathbb{R},(u,v)\mapsto A\int_Xu\omega_v^{\prime n}
			\end{align*}where the first map is $d_1$-continuous and sends $d_{1,\omega}$-bounded subsets to $d_{1,\omega'}$-bounded subsets by Lemma \ref{extensionLemma1}.
			
			Let $u_i\overset{d_{1,\omega}}{\to}u$ and $v_j\overset{d_{1,\omega}}{\to}v$ be sequences in $\mathcal{E}^1(X,\omega)$. Assume first that $\sup_Xu=\sup_Xv=\sup_Xu_i=\sup_Xv_j=0$. Recall that $\operatorname{E}_\omega$ is bounded on $d_1$-bounded subsets (see \cite[Proposition 4.1]{BDL2017GT}), so $u,v,u_i,v_j\in \mathcal{E}^1_C(X,\omega)$, c.f. \eqref{E_C}. 
			
			Now by \cite[Corollary 5.8]{Darvas2} and its proof, for each $R>0$ there exists
			a continuous function $f_R: \mathbb{R}\to\mathbb{R}$ with $f_R(0)=0$ such that
			\begin{align*}
				\int_X|u-v|\,\omega_h^n\leq f_R(d_{1,\omega}(u,v))
			\end{align*}for any $u,v,h\in\mathcal{E}^1(X,\omega)\cap\{w:d_{1,\omega}(0,w)\leq R\}.$ Take\begin{align*}
			R=\max\{d_{1,\omega}(0,u_i),d_{1,\omega}(0,u),d_{1,\omega}(0,v_j),d_{1,\omega}(0,v)\}+1,
			\end{align*}then we have\begin{align*}
				\int_X|u-u_i|\omega_v^n\leq f_R(d_{1,\omega}(u,u_i)).
			\end{align*}
		 	
		 	On the other hand, by Proposition \ref{BBGZLemma3.13}, we have \begin{align*}
		 		\sup_{\varphi\in\mathcal{E}^1_C(X,\omega)}\left|\int_X\varphi(\omega^n_\psi-\omega^n_{\psi'})\right|=O(I(\psi,\psi')^{1/2})
		 	\end{align*}uniformly for $\psi,\psi'\in \mathcal{E}^1_C$. The conclusion is then implied at this time, since\begin{align}\label{foramonent}
		 	\left|\int_Xu\omega_v^{n}-\int_Xu_i\omega_{v_j}^{n}\right|\leq \int_X|u-u_i|\omega_v^n+\left|\int_Xu_i(\omega_{v_j}^{n}-\omega_v^n)\right|\leq f_R(d_{1,\omega}(u,u_i))+O(I_{\omega}(v_j,v)^{1/2})
		 	\end{align}and $I_{\omega}\leq \mathbf{I}_{1,\omega}$.
	 	
	 		For general case, let \begin{align*}
	 			\tilde{u}_i=u_i-\sup_Xu_i,\tilde{u}=u-\sup_Xu,\tilde{v}_j=v_j-\sup_Xv_j,\tilde{v}=v-\sup_Xv.
	 			\end{align*} Since $d_1$-convergence implies $L^1$-convergence (see for example \cite[Theorem 5]{Darvas2}), by Hartogs Lemma (see for example \cite[Proposition 8.4 (2)]{GZ17}), we have $\sup_Xu_i\to \sup_Xu,\sup_Xv_j\to \sup_Xv$. Thus\begin{align*}
	 			d_1(u_i-\sup_Xu_i,u-\sup_Xu) \to 0,\ d_1(v_j-\sup_Xv_j,v-\sup_Xv)\to 0,
	 			\end{align*}since\begin{align*}
	 			u_i-\sup_Xu_i\overset{L^1}{\to}u-\sup_Xu,\ v_j-\sup_Xv_j\overset{L^1}{\to}v-\sup_Xv
	 			\end{align*}and\begin{align*}
	 			&\operatorname{E}_\omega(u_i-\sup_Xu_i)=\operatorname{E}_\omega(u_i)-\frac{V_\omega}{n+1}\sup_Xu_i\to \operatorname{E}_\omega(u)-\frac{V_\omega}{n+1}\sup_Xu=\operatorname{E}_\omega(u-\sup_Xu),\\
	 			&\operatorname{E}_\omega(v_j-\sup_Xv_j)=\operatorname{E}_\omega(v_j)-\frac{V_\omega}{n+1}\sup_Xv_j\to \operatorname{E}_\omega(v)-\frac{V_\omega}{n+1}\sup_Xv=\operatorname{E}_\omega(v-\sup_Xv).
 			\end{align*}So we have\begin{align*}
 				\int_Xu_i\omega_{v_i}^n=\int_X\tilde{u}_i\omega_{\tilde{v}_i}^n+V_{\omega}\sup_Xu_i\to \int_X\tilde{u}\omega_{\tilde{v}}^n+V_{\omega}\sup_Xu=\int_Xu\omega_v^n.
 			\end{align*}\indent Now assume $W$ is a $d_{1,\omega}\times d_{1,\omega}$-bounded subset, by triangle inequality $\sup_{(u,v)\in W}\{d_{1,\omega}(0,u)+d_{1,\omega}(0,v)\}<C$. By \cite[Corollary 4]{Darvas2}, we know that $\left|\sup_Xu\right|$ and $\left|\sup_Xv\right|$ are uniformly bounded for $(u,v)\in W$. Let $\tilde{u}=u-\sup_Xu$ and $\tilde{v}=v-\sup_Xv$. Then\begin{align*}
 			&d_{1,\omega}(\tilde{u},0)\leq 	d_{1,\omega}(u-\sup_Xu,-\sup_Xu)+d_{1,\omega}(-\sup_Xu,0)\leq d_{1,\omega}(0,u)+\left|\sup_Xu\right|,\\
 			&d_{1,\omega}(\tilde{v},0)\leq 	d_{1,\omega}(v-\sup_Xv,-\sup_Xv)+d_{1,\omega}(-\sup_Xv,0)\leq d_{1,\omega}(0,v)+\left|\sup_Xv\right|
 			\end{align*}are uniformly bounded. Now by (\ref{foramonent}), we have\begin{align*}
 				\left|\int_Xu\omega_v^n\right|\leq V_{\omega}\left|\sup_Xu\right|+	\left|\int_X\tilde{u}\omega_{\tilde{v}}^n\right|\leq V_{\omega}\left|\sup_Xu\right|+ f_R(d_{1,\omega}(u,0))+O(I_\omega(v,0)^{1/2})
 			\end{align*}for a suitable big $R$. This means $\int_Xu\omega_v^n$ is bounded on $W$.\end{proof}
		By virtue of the above lemmas, we can prove Theorem \ref{finiteenergyextension} now.
		\begin{proof}[Proof of Theorem \ref{finiteenergyextension}]
			By \cite[Proposition 4.1, Proposition 4.4, Theorem 4.7]{BDL2017GT}, $\operatorname{E}_{\omega_i}(u_i)$ is $d_{1,\omega_i}$-continuous, $\operatorname{E}^{\omega_i}_{\omega_0}(u_0)$ is $d_{1,\omega_0}$-continuous and $\mathcal{M}_{\operatorname{cscK}}(\omega_{0}+dd^cu_0,\omega_{0})$ is the greatest $d_1$-lsc extension of $\mathcal{M}_{\operatorname{cscK}}|_{\mathcal{H}_{\omega_0}}$. Thus $\mathcal{M}_{\operatorname{ccscK}}|_{\prod_{i=0}^m\mathcal{E}^1(X,\omega_i)}$ is the greatest $d_1$-lsc extension of (\ref{coupledMabuchiDef}) by Lemma \ref{extensionLemma2}.
			
			Now we prove convexity. For any tuple of finite energy geodesics $(u_i(t))_{i=0}^m\subset \prod_{i=0}^m\mathcal{E}^1(X,\omega_i)$, by the definition of finite energy geodesic (see \cite{Darvas2}), there exists weak geodesics $(u_i^k(t))_{i=0}^m$ with smooth ends such that\begin{align*}
				d_{1,\omega_i}(u_i^k(0),u_i(0))\to 0,d_{1,\omega_i}(u_i^k(1),u_i(1))\to 0,k\to\infty.
			\end{align*}By endpoint stability of ﬁnite energy geodesic \cite[Proposition 4.3]{BDL2017GT}, there exists subsequence $k_j$ such that \begin{align*}
			d_{1,\omega_i}(u_i^{k_j}(t),u_i(t))\to 0,k_j\to\infty.
			\end{align*}Set $E(u_0,\cdots,u_m):=\mathcal{M}_{\operatorname{ccscK}}(u_0,\cdots,u_m)- \mathcal{M}_{\operatorname{cscK}}(\omega_{0}+dd^cu_0,\omega_{0})$, which is $\prod_{i=0}^m\ d_{1,\omega_i}$-continous. By Proposition \ref{ddcofanalyticDelignepairingweak} (iii), $E$ is convex along $(u_i^{k_j}(t))_{i=0}^m$, thus as a limit of convex functions, $E\left(u_0(t),\cdots,u_m(t)\right)$ is convex in $[0,1]$. The Theorem is followed by \cite[Theorem 1.2]{BDL2017GT}.
		\end{proof}
		
		\begin{remark}
			Note that our theorem permits the following possibility:\begin{align*}
				\mathcal{M}_{\operatorname{ccscK}}(u_0(t),\cdots,u_m(t))=\begin{cases}
					0, \ t\in [0,1/2]\\
					\infty, \ t\in (1/2,1],
				\end{cases}
			\end{align*}for some tuple of finite-energy geodesics in $\prod_{i=0}^m\mathcal{E}^1(X,\omega_i)$. The proof of \cite[Theorem 1.2]{BDL2017GT} cannot rule out this possibility because $\mathcal{M}_{\operatorname{ccscK}}$ is allowed to take $\infty$ value. Actually  $\mathcal{M}_{\operatorname{ccscK}}$ is continuous, finite-valued and convexity along any tuple of finite-energy geodesics in $\operatorname{Ent}(X,\omega_0)\times\prod_{i=1}^m\mathcal{E}^1(X,\omega_i)$. Here $\operatorname{Ent}(X,\omega_0)=\{u_0\in \mathcal{E}^1(X,\omega_0):\operatorname{Ent}((\omega_0+dd^cu_0)^n,\omega_0^n)<\infty\}$ is the space of ﬁnite-entropy potentials, which is a $d_{1,\omega_0}$-geodesically convex subspace in $\mathcal{E}^1(X,\omega_0)$ (see \cite[Section 4.4]{BDL2017GT}).
		\end{remark}
		\boldmath
		\subsection{On the functional $\mathcal{M}'$}\color{black}
		\unboldmath
		\label{subsec:6}
		Under the assumption $c_1(X)=\sum_{i=0}^{m}[\omega_i]$ derived from the cKE equation (\ref{cKE1}), consider the following functional\begin{equation}\label{M'}
			\begin{aligned}
				&\mathcal{M}'((\omega_i(s)),(\omega_i(0))) :=-\sum_{i=0}^m\int_0^s\int_X\dot{\varphi_i}(t)\frac{1}{V_i} \frac{\omega_i(t)^n}{n!}\wedge dt \\& + \frac{1}{V}\int_0^s\int_X \sum_{i=0}^m\dot{\varphi_i}(t) (n+1-S_{\sum_{i=0}^m\omega_i(t),J})\frac{\left(\sum_{i=0}^m\omega_i(t)\right)^n}{n!}\wedge dt,
			\end{aligned}
		\end{equation}where $V=(-K_X)^n$. We can apply all the content in this section to this functional and get
		\begin{proposition}\label{our5}\mbox{}\\
			$(\mathrm{i})$ $\mathcal{M}'$ is well-defined and\begin{align*}
				&\mathcal{M}'((\omega_i+dd^c\varphi_i),(\omega_i))= \mathcal{M}_{\operatorname{cscK}}\left(\sum_{i=0}^m(\omega_i+dd^c\varphi_i),\sum_{i=0}^m\omega_i\right)\\&-\sum_{i=0}^m\frac{1}{V_i(n+1)!}\braket{\varphi_{i},\cdots,\varphi_{i}}_{(\omega_i,\cdots,\omega_i)}+\frac{1}{V(n+1)!}\bigg<\sum_{i=0}^m\varphi_{i},\cdots,\sum_{i=0}^m\varphi_{i}\bigg>_{\left(\sum_{i=0}^m\omega_i,\cdots,\sum_{i=0}^m\omega_i\right)};
			\end{align*}\noindent $(\mathrm{ii})$ Let $X$ be a complex manifold with ample line bundles $L_0,\cdots,L_m$ on it, and  $(\mathcal{X},\mathcal{L}_i)$ be ample algebraic test conﬁgurations for $(X,L_i)$ respectively. Suppose that $\mathcal{X}$ is smooth and dominates $X\times\mathbb{P}^1$. Then for any smooth strictly positive metrics $\Phi_i$ on $\mathcal{L}_i$ near the central ﬁber, we have
			\begin{equation}\label{slopeforHWNDF}
				\begin{aligned}
					\lim_{t\to\infty}\frac{\mathcal{M}'\left((\omega_i(t)),(\omega_i)\right)}{t}&=-\sum_{i=0}^m\frac{1}{V_i(n+1)!}(\mathcal{L}_{i}^{n+1})+\frac{1}{Vn!}\left(\left(K_{\mathcal{X}/\mathbb{P}^1}+\sum_{i=0}^m\mathcal{L}_i\right)\cdot\left(\sum_{i=0}^m\mathcal{L}_i\right)^n\right)\\&\ +\frac{1}{Vn!}\left((\mathcal{X}(0)_{\operatorname{red}}-\mathcal{X}(0))\cdot \left(\sum_{i=0}^m\mathcal{L}_i\right)^n\right)
				\end{aligned}
			\end{equation}where $\omega_i$ is a reference metric in $c_1(L_i)$, $\omega_i(t) = dd^c\phi_i(t)$, $\phi_i(t)$ is the ray of smooth positive metrics on $L_i$ corresponding to $\Phi_i$.
		\end{proposition}
		\begin{proof}
			By (\ref{M'}),\begin{align*}
				&\mathcal{M}'((\omega_i(s)),(\omega_i(0))) :=\frac{1}{V}\int_0^s\int_M \sum_{i=0}^m\dot{\varphi_i}(t) (n-S_{\sum_{i=0}^m\omega_i(t),J})\frac{\left(\sum_{i=0}^m\omega_i(t)\right)^n}{n!}\wedge dt\\&-\sum_{i=0}^m\int_0^s\int_X\dot{\varphi_i}(t)\frac{1}{V_i} \frac{\omega_i(t)^n}{n!}\wedge dt +\frac{1}{V}\int_0^s\int_M \sum_{i=0}^m\dot{\varphi_i}(t)\frac{\left(\sum_{i=0}^m\omega_i(t)\right)^n}{n!}\wedge dt.
			\end{align*}The first term is just the classical K-energy functional $\mathcal{M}_{\operatorname{cscK}}\left(\sum_{i=0}^m(\omega_i+dd^c\varphi_i),\sum_{i=0}^m\omega_i\right)$, which is independent of the choice of path. For the rest, by (\ref{MongeAmpere}), we have\begin{small}
				\begin{align*}
					&\int_0^s\int_X\dot{\varphi_i}(t)\frac{1}{V_i} \frac{\omega_i(t)^n}{n!}\wedge dt=\frac{1}{V_i(n+1)!}\braket{\varphi_i(s),\cdots,\varphi_i(s)}_{(\omega_i,\cdots,\omega_i)}\ \text{and}\\
					&\frac{1}{V}\int_0^s\int_M \sum_{i=0}^m\dot{\varphi_i}(t)\frac{\left(\sum_{i=0}^m\omega_i(t)\right)^n}{n!}\wedge dt=\frac{1}{V(n+1)!}\bigg<\sum_{i=0}^m\varphi_{i}(s),\cdots,\sum_{i=0}^m\varphi_{i}(s)\bigg>_{\left(\sum_{i=0}^m\omega_i,\cdots,\sum_{i=0}^m\omega_i\right)},
				\end{align*}
			\end{small}thus $\mathcal{M}'$ is independent of choice of path in $\prod_{i=0}^m\mathcal{H}(\omega_i)$. Since $\mathcal{M}_{\operatorname{cscK}}$ is invariant under the
			addition of constants, it remains only to examine the analytic Deligne pairing terms. Using
			\begin{align*}
				\braket{\varphi+C,\cdots,\varphi+C}_{(\omega,\cdots,\omega)}
				=
				\braket{\varphi,\cdots,\varphi}_{(\omega,\cdots,\omega)}+C(n+1)\int_X\omega^n,
			\end{align*}
			we see that the resulting constant contributions cancel in the
			definition of $\mathcal{M}'$. Hence,
			\begin{align*}
				\mathcal{M}'(\varphi_0+C_0,\cdots,\varphi_m+C_m)
				=
				\mathcal{M}'(\varphi_0,\cdots,\varphi_m),
			\end{align*}
			so $\mathcal{M}'$ descends to
			$\prod_{i=0}^m\mathcal{K}([\omega_i])$. This proves (i).
			
			The proof of (ii) is exactly same as Theorem \ref{our2alg}. At this time, the smooth strictly positive metrics $\Phi_i$ induces a smooth strictly positive metric $\sum_{i=0}^m\Phi_i$ on $\sum_{i=0}^m\mathcal{L}_i$. By \cite[Theorem A]{boucksom2019}\begin{align*}
				&\lim_{t\to\infty}\frac{\mathcal{M}_{\operatorname{cscK}}\left(\sum_{i=0}^m(\omega_i+dd^c\varphi_i(t)),\sum_{i=0}^m\omega_i\right)}{t}\\=&\ \frac{1}{Vn!}\left((K_{\mathcal{X}/\mathbb{P}}+\mathcal{X}(0)_{\operatorname{red}}-\mathcal{X}(0))\cdot \left(\sum_{i=0}^m\mathcal{L}_i\right)^n\right)+\frac{n}{V(n+1)!}\left(\left(\sum_{i=0}^m\mathcal{L}_i\right)^{n+1}\right).
			\end{align*}Since the definition of analytic Deligne pairing identifies with the Deligne pairing defined in \cite[Section 1.3]{boucksom2019} by \cite[(1.5)]{boucksom2019}. The proof is thus completed by applying above and invoking \cite[Lemma 3.9]{boucksom2019}.
		\end{proof}
		
		\begin{remark}\label{remarkM'}
			The slope (\ref{slopeforHWNDF}) is just the Donaldson-Futaki invariant $\operatorname{DF}(\mathcal{X},(\mathcal{L}_i))$ given by \cite{hultgren2019coupled} (see also \cite[Proposition 19]{datar2020coupled}). This DF invariant differs from our asymptotic slope (\ref{ourslopealg}) for two main reasons. First, the coupled K-energy and the ccscK equation single out $\omega_0$, making them less symmetric than the cKE setting. Second, \cite{hultgren2019coupled} and \cite{datar2020coupled} employ a $\mathbb{Q}$-decomposition test configuration $(\mathcal{X},(\mathcal{L}_i))$ that shares a common total space $\mathcal{X}$. In contrast to our tuple of test configurations $(\mathcal{X}_i,\mathcal{L}_i)$, their configuration is specifically designed to compute the slope of $\mathcal{M}_{\operatorname{cscK}}\left(\sum_{i=0}^m\omega_i+dd^c\varphi_i,\sum_{i=0}^m\omega_i\right)$. 
			It is not clear whether both $\mathcal{M}_{\operatorname{ccscK}}$ and $\mathcal{M}'$ are equally effective in the variational approach for the cKE equation. 
			The critical point of $\mathcal{M}'$ should be\begin{align*}
				\frac{1}{V}(n+1-S_{\sum_{i=0}^m\omega_i,J})\left(\sum_{i=0}^m\omega_i\right)^n=\frac{\omega_i^n}{V_i}.
			\end{align*}It is not so clear how the cKE equation is related to this equation. The convexity of $\mathcal{M}'$ is also unclear. 
			Actually we will see in the next section that\begin{align*}
				&\mathcal{M}'((\omega_{i}+dd^c\varphi_i),(\omega_{i}))-\mathcal{N}_{\operatorname{cKE}}((\omega_{i}+dd^c\varphi_i),(\omega_{i}))\\=&\  \mathcal{M}_{\operatorname{cscK}}\left(\sum_{i=0}^m(\omega_i+dd^c\varphi_i),\sum_{i=0}^m\omega_i\right)-\mathcal{N}_{\operatorname{KE}}\left(\sum_{i=0}^m(\omega_i+dd^c\varphi_i),\sum_{i=0}^m\omega_i\right)
			\end{align*}
			where $\mathcal{N}_{\operatorname{cKE}}$ is defined in the next subsection and $\mathcal{N}_{\operatorname{KE}}$ is its uncoupled case.
	\end{remark}
	
	\boldmath
	\subsection{An auxiliary energy functional for coupled KE equation}
	\unboldmath
	\label{subsec:3}
	
	As mentioned above, the +1 case cKE equation (\ref{cKE1}) does not always have a solution, just like the classical K\"{a}hler-Einstein problem. To find a condition for the existence of solutions, a variational approach using the following useful energy functional was considered.

\begin{definition}[\hspace*{-0.6ex}\text{\cite[Deﬁnition 2.4]{hultgren2019coupled}}]\label{NcKE}
	Let $\omega_i\in \alpha_i$ be as above. Since $\sum_{i=0}^m\omega_i\in c_1(X)$, by Yau’s theorem we can ﬁnd a K\"{a}hler form
$\Omega$ in $c_1(X)$ such that $\Omega:\operatorname{Ric}(\Omega)=\sum_{i=0}^m\omega_i$. We deﬁne an energy functional $\mathcal{N}_{\operatorname{cKE}}$ on $\prod_{i=0}^m\mathcal{K}([\omega_i])$ by\begin{equation}\label{Ding}
		\begin{aligned}
			&\mathcal{N}_{\operatorname{cKE}}((\omega_i+dd^c\varphi_i),(\omega_i)):=-\sum_{i=0}^m\frac{1}{V_i}\operatorname{E}_{\omega_i}(\varphi_i)-\frac{1}{n!}\log\left(\int_Xe^{-\sum_{i=0}^m\varphi_i}\Omega^n\right).
		\end{aligned}
	\end{equation}
\end{definition}

The functional $\mathcal{N}_{\operatorname{cKE}}$ can be extended to the finite energy class $\prod_{i=0}^{m}\mathcal{E}^1(\omega_i)$ (see e.g., \cite[Deﬁnition 10.13]{GZ17} for the definition), and has the following properties:

\begin{theorem}[\normalfont{\small\hspace*{-0.6ex}\text{\cite[Theorem 4.3, Lemma 4.2, Theorem 2.5]{hultgren2019coupled}, \cite[Theorem 1.2]{Takahashi}, \cite[Theorem 1.3]{fujita2024coupleddingstabilityyautiandonaldson}}}]\label{DingBoundedBelow}\mbox{}$(\mathrm{i})$ $\mathcal{N}_{\operatorname{cKE}}$ is convex along weak geodesics;\\
	$(\mathrm{ii})$ Any minimizer of $\mathcal{N}_{\operatorname{cKE}}$ is smooth, K\"{a}hler and cKE;\\
	$(\mathrm{iii})$ If $([\omega_i])$ admits a coupled K\"{a}hler–Einstein metric, then $\mathcal{N}_{\operatorname{cKE}}$ is bounded below on $\prod_{i=0}^m\mathcal{K}([\omega_i])$. Actually, this existence is equivalent to a certain coercivity condition on $\mathcal{N}_{\operatorname{cKE}}$ and also to a uniform stability condition
		(see \cite[Theorem 1.3]{fujita2024coupleddingstabilityyautiandonaldson} for the precise statement).
\end{theorem}

\boldmath
\section{Relationship between $\mathcal{M}_{\operatorname{ccscK}},\ \mathcal{M}'$ and $\mathcal{N}_{\operatorname{cKE}}$ on Fano setting}\label{sec:4}
\unboldmath
	\subsection{Relation with entropy}
	The relative entropy $\operatorname{Ent}(\mu,\nu)$ for probability measure $\mu,\nu$ is defined by
	\begin{align}\label{Entropy}
		\operatorname{Ent}(\mu,\nu)=\frac{1}{n!}\int_X\log\left(\frac{d\mu}{d\nu}\right)d\mu.
	\end{align}
	Then 
	\begin{align}\label{Entropy2}
		\operatorname{Ent}(\mu,\nu) \ge 0
	\end{align}
	because for the convex function $g = x\log x$ we use Jensen's inequality with $f = \frac{d\mu}{d\nu}$ to show
	\begin{eqnarray*}
		\int_X \log\left(\frac{d\mu}{d\nu}\right)\,d\mu &=& \int_X \log\left(\frac{d\mu}{d\nu}\right)\,\frac{d\mu}{d\nu}\,d\nu = \int_X g\circ f d\nu\\
		&\ge& g\left( \int_X f d\nu  \right) = g\left(\int_X d\nu\right) = g(1) = 0.
	\end{eqnarray*}
\begin{theorem}Suppose that $c_1(X)=\sum_{i=0}^m[\omega_i]$ and 
let $\Omega$ be the K\"{a}hler form in $c_1(X)$ such that $\Omega:\operatorname{Ric}(\Omega)=\sum_{i=0}^m\omega_i$.
Then
	\begin{equation}
		\begin{aligned}
			&\mathcal{M}_{\operatorname{ccscK}}((\omega_{i}+dd^c\varphi_i),(\omega_{i}))\\&=\mathcal{N}_{\operatorname{cKE}}((\omega_{i}+dd^c\varphi_i),(\omega_{i}))+\operatorname{Ent}\left(\frac{(\omega_0+dd^c\varphi_0)^n}{\int_X\omega_0^n},\frac{e^{-\sum_{i=0}^m\varphi_i}\Omega^n}{\int_Xe^{-\sum_{i=0}^m\varphi_i}\Omega^n}\right)-R,
		\end{aligned}
	\end{equation}holds for any $\varphi_i\in \mathcal{H}(\omega_i)$, where $R:=\operatorname{Ent}\left(\frac{\omega_0^n}{\int_X\omega_0^n},\frac{\Omega^n}{\int_X\Omega^n}\right)$. In particular $\mathcal{M}_{\operatorname{ccscK}}((\omega_{i}+dd^c\varphi_i),(\omega_{i}))\geq\mathcal{N}_{\operatorname{cKE}}((\omega_{i}+dd^c\varphi_i),(\omega_{i}))-R$.
\end{theorem}
\begin{proof}
	Note that under the topological assumption $c_1(X)=\sum_{i=0}^m[\omega_i]$, we have\begin{align*}
		&\frac{d}{dt}\bigg|_{t=0}\mathcal{M}_{\operatorname{ccscK}}\left((\omega_{i}+dd^c\varphi_i+tdd^cv_i),(\omega_{i})\right)\\&=-\sum_{i=1}^m\int_Xv_i\left(\frac{1}{V_i} \frac{(\omega_i+dd^c\varphi_i)^n}{n!}- \frac{1}{V_0}\frac{(\omega_0+dd^c\varphi_0)^n}{n!} \right) \\& - \frac{1}{V_0}\int_M v_0 \left(S_{\omega_0+dd^c\varphi_0,J}\frac{(\omega_0+dd^c\varphi_0)^n}{n!} - \sum_{i=0}^m(\omega_i+dd^c\varphi_i) \wedge \frac{(\omega_0+dd^c\varphi_0)^{n-1}}{(n-1)!} \right).
	\end{align*}So it is equivalent to show that\begin{align*}
		&-\sum_{i=1}^m\int_Xv_i\left(\frac{1}{V_i} \frac{(\omega_i+dd^c\varphi_i)^n}{n!}- \frac{1}{V_0}\frac{(\omega_0+dd^c\varphi_0)^n}{n!} \right) 
		\\& - \frac{1}{V_0}\int_M v_0 \left(S_{\omega_0+dd^c\varphi_0,J}\frac{(\omega_0+dd^c\varphi_0)^n}{n!} 
		- \sum_{i=0}^m(\omega_i+dd^c\varphi_i) \wedge \frac{(\omega_0+dd^c\varphi_0)^{n-1}}{(n-1)!} \right)\\&+\sum_{i=0}^m\int_Xv_i\frac{1}{V_i} \frac{(\omega_i+dd^c\varphi_i)^n}{n!}+\frac{1}{n!}\frac{d}{dt}\bigg|_{t=0}\log\left(\int_Xe^{-\sum_{i=0}^m(\varphi_i+tv_i)}\Omega^n\right)
	\end{align*}equals to  $\frac{d}{dt}\bigg|_{t=0}\operatorname{Ent}\left(\frac{(\omega_0+dd^c\varphi_0+tdd^cv_0)^n}{\int_X\omega_0^n},\frac{e^{-\sum_{i=0}^m(\varphi_i+tv_i)}\Omega^n}{\int_Xe^{-\sum_{i=0}^m(\varphi_i+tv_i)}\Omega^n}\right)$, i.e.\begin{align*}
		&\frac{d}{dt}\bigg|_{t=0}\operatorname{Ent}\left(\frac{(\omega_0+dd^c\varphi_0+tdd^cv_0)^n}{\int_X\omega_0^n},\frac{e^{-\sum_{i=0}^m(\varphi_i+tv_i)}\Omega^n}{\int_Xe^{-\sum_{i=0}^m(\varphi_i+tv_i)}\Omega^n}\right)
		\\&=\sum_{i=0}^m\int_Xv_i\frac{1}{V_0} \frac{(\omega_0+dd^c\varphi_0)^n}{n!}-\frac{\sum_{i=0}^m\int_Xv_ie^{-\sum_{i=0}^m\varphi_i}\Omega^n}{n!\int_Xe^{-\sum_{i=0}^m\varphi_i}\Omega^n}
		\\& - \frac{1}{V_0}\int_M v_0 \left(S_{\omega_0+dd^c\varphi_0,J}\frac{(\omega_0+dd^c\varphi_0)^n}{n!} 
		- \sum_{i=0}^m(\omega_i+dd^c\varphi_i) \wedge \frac{(\omega_0+dd^c\varphi_0)^{n-1}}{(n-1)!} \right).
	\end{align*}The left hand side is\begin{align*}
		&\frac{d}{dt}\bigg|_{t=0}\operatorname{Ent}\left(\frac{(\omega_0+dd^c\varphi_0+tdd^cv_0)^n}{\int_X\omega_0^n},\frac{e^{-\sum_{i=0}^m(\varphi_i+tv_i)}\Omega^n}{\int_Xe^{-\sum_{i=0}^m(\varphi_i+tv_i)}\Omega^n}\right)
		\\&=\frac{d}{dt}\bigg|_{t=0}\int_X\log\left(\frac{(\omega_0+dd^c\varphi_0+tdd^cv_0)^n}{\int_X\omega_0^n}\bigg/\frac{e^{-\sum_{i=0}^m(\varphi_i+tv_i)}\Omega^n}{\int_Xe^{-\sum_{i=0}^m(\varphi_i+tv_i)}\Omega^n}\right)\frac{(\omega_0+dd^c\varphi_0+tdd^cv_0)^n}{n!\int_X\omega_0^n}
		\\&=\frac{d}{dt}\bigg|_{t=0}\int_X\log\left(\frac{(\omega_0+dd^c\varphi_0+tdd^cv_0)^n}{\int_X\omega_0^n}\bigg/\frac{e^{-\sum_{i=0}^m(\varphi_i+tv_i)}\Omega^n}{\int_Xe^{-\sum_{i=0}^m(\varphi_i+tv_i)}\Omega^n}\right)\frac{(\omega_0+dd^c\varphi_0)^n}{n!\int_X\omega_0^n}\\&\quad + \frac{d}{dt}\bigg|_{t=0}\int_X\log\left(\frac{(\omega_0+dd^c\varphi_0)^n}{\int_X\omega_0^n}\bigg/\frac{e^{-\sum_{i=0}^m\varphi_i}\Omega^n}{\int_Xe^{-\sum_{i=0}^m\varphi_i}\Omega^n}\right)\frac{(\omega_0+dd^c\varphi_0+tdd^cv_0)^n}{n!\int_X\omega_0^n}.
	\end{align*}Note that\begin{align*}
		&\frac{d}{dt}\bigg|_{t=0}\log\left(\frac{(\omega_0+dd^c\varphi_0+tdd^cv_0)^n}{\int_X\omega_0^n}\bigg/\frac{e^{-\sum_{i=0}^m(\varphi_i+tv_i)}\Omega^n}{\int_Xe^{-\sum_{i=0}^m(\varphi_i+tv_i)}\Omega^n}\right)\\&=\frac{d}{dt}\bigg|_{t=0}\log\left(\frac{(\omega_0+dd^c\varphi_0+tdd^cv_0)^n}{e^{-\sum_{i=0}^m(\varphi_i+tv_i)}\Omega^n}\right)+\frac{d}{dt}\bigg|_{t=0}\log\left(\int_Xe^{-\sum_{i=0}^m(\varphi_i+tv_i)}\Omega^n\right)\\&=\frac{d}{dt}\bigg|_{t=0}\left[\log\left(\frac{(\omega_0+dd^c\varphi_0+tdd^cv_0)^n}{\Omega^n}\right)+\sum_{i=0}^m(\varphi_i+tv_i)+\log\left(\int_Xe^{-\sum_{i=0}^m(\varphi_i+tv_i)}\Omega^n\right)\right]\\&=\sum_{i=0}^mv_i+\frac{\Omega^n}{(\omega_0+dd^c\varphi_0)^n}\frac{ndd^cv_0\wedge(\omega_0+dd^c\varphi_0)^{n-1}}{\Omega^n}-\frac{\int_X\left(\sum_{i=0}^mv_i\right)e^{-\sum_{i=0}^m\varphi_i}\Omega^n}{\int_Xe^{-\sum_{i=0}^m\varphi_i}\Omega^n}\\&=\sum_{i=0}^mv_i+\frac{ndd^cv_0\wedge(\omega_0+dd^c\varphi_0)^{n-1}}{(\omega_0+dd^c\varphi_0)^n}-\frac{\int_X\left(\sum_{i=0}^mv_i\right)e^{-\sum_{i=0}^m\varphi_i}\Omega^n}{\int_Xe^{-\sum_{i=0}^m\varphi_i}\Omega^n},
	\end{align*}thus\begin{align*}
		&\frac{d}{dt}\bigg|_{t=0}\operatorname{Ent}\left(\frac{(\omega_0+dd^c\varphi_0+tdd^cv_0)^n}{\int_X\omega_0^n},\frac{e^{-\sum_{i=0}^m(\varphi_i+tv_i)}\Omega^n}{\int_Xe^{-\sum_{i=0}^m(\varphi_i+tv_i)}\Omega^n}\right)
		\\&=\frac{d}{dt}\bigg|_{t=0}\int_X\log\left(\frac{(\omega_0+dd^c\varphi_0+tdd^cv_0)^n}{\int_X\omega_0^n}\bigg/\frac{e^{-\sum_{i=0}^m(\varphi_i+tv_i)}\Omega^n}{\int_Xe^{-\sum_{i=0}^m(\varphi_i+tv_i)}\Omega^n}\right)\frac{(\omega_0+dd^c\varphi_0)^n}{n!\int_X\omega_0^n}\\&\quad + \int_X\log\left(\frac{(\omega_0+dd^c\varphi_0)^n}{\int_X\omega_0^n}\bigg/\frac{e^{-\sum_{i=0}^m\varphi_i}\Omega^n}{\int_Xe^{-\sum_{i=0}^m\varphi_i}\Omega^n}\right)\frac{dd^cv_0\wedge(\omega_0+dd^c\varphi_0)^{n-1}}{(n-1)!\int_X\omega_0^n}
		\\&=\int_X\left(\sum_{i=0}^mv_i+\frac{ndd^cv_0\wedge(\omega_0+dd^c\varphi_0)^{n-1}}{(\omega_0+dd^c\varphi_0)^n}-\frac{\int_X\left(\sum_{i=0}^mv_i\right)e^{-\sum_{i=0}^m\varphi_i}\Omega^n}{\int_Xe^{-\sum_{i=0}^m\varphi_i}\Omega^n}\right)\frac{(\omega_0+dd^c\varphi_0)^n}{n!\int_X\omega_0^n}\\&\quad  + \int_Xv_0dd^c\log\left(\frac{(\omega_0+dd^c\varphi_0)^n}{\int_X\omega_0^n}\bigg/\frac{e^{-\sum_{i=0}^m\varphi_i}\Omega^n}{\int_Xe^{-\sum_{i=0}^m\varphi_i}\Omega^n}\right)\wedge\frac{(\omega_0+dd^c\varphi_0)^{n-1}}{(n-1)!\int_X\omega_0^n}
		\\&=\sum_{i=0}^m\int_Xv_i\frac{(\omega_0+dd^c\varphi_0)^n}{n!\int_X\omega_0^n}+\int_X\frac{dd^cv_0\wedge(\omega_0+dd^c\varphi_0)^{n-1}}{(n-1)!\int_X\omega_0^n}-\frac{\int_X\left(\sum_{i=0}^mv_i\right)e^{-\sum_{i=0}^m\varphi_i}\Omega^n}{n!\int_Xe^{-\sum_{i=0}^m\varphi_i}\Omega^n}\\&\quad + \int_Xv_0\left(\operatorname{Ric}(\Omega)-\operatorname{Ric}(\omega_0+dd^c\varphi_0)+\sum_{i=0}^mdd^c\varphi_i\right)\wedge\frac{(\omega_0+dd^c\varphi_0)^{n-1}}{(n-1)!V_0}
		\\&=\sum_{i=0}^m\int_Xv_i\frac{(\omega_0+dd^c\varphi_0)^n}{n!V_0}-\frac{\int_X\left(\sum_{i=0}^mv_i\right)e^{-\sum_{i=0}^m\varphi_i}\Omega^n}{n!\int_Xe^{-\sum_{i=0}^m\varphi_i}\Omega^n}\\&\quad + \int_Xv_0\left(\sum_{i=0}^m\omega_i-\operatorname{Ric}(\omega_0+dd^c\varphi_0)+\sum_{i=0}^mdd^c\varphi_i\right)\wedge\frac{(\omega_0+dd^c\varphi_0)^{n-1}}{(n-1)!V_0}
		\\&=\sum_{i=0}^m\int_Xv_i\frac{(\omega_0+dd^c\varphi_0)^n}{n!V_0}-\frac{\int_X\left(\sum_{i=0}^mv_i\right)e^{-\sum_{i=0}^m\varphi_i}\Omega^n}{n!\int_Xe^{-\sum_{i=0}^m\varphi_i}\Omega^n}\\&\quad + \sum_{i=0}^m\int_Xv_0\left(\omega_i+dd^c\varphi_i\right)\wedge\frac{(\omega_0+dd^c\varphi_0)^{n-1}}{(n-1)!V_0}-\int_Xv_0S_{\omega_0+dd^c\varphi_0,J}\frac{(\omega_0+dd^c\varphi_0)^{n}}{n!V_0}.
	\end{align*}The first statement of the theorem follows from this. The second statement follows from (\ref{Entropy2}).
\end{proof}The same method can be used to prove the following identity, which can also be derived from the relationship between $\mathcal{M}_{\operatorname{cscK}}$ and $\mathcal{N}_{\operatorname{KE}}$ in the classical setting.
\begin{theorem}Suppose that $c_1(X)=\sum_{i=0}^m[\omega_i]$ and 
let $\Omega$ be the K\"{a}hler form in $c_1(X)$ such that $\Omega:\operatorname{Ric}(\Omega)=\sum_{i=0}^m\omega_i$.
Then
	\begin{align*}
		&\mathcal{M}'((\omega_{i}+dd^c\varphi_i),(\omega_{i}))\\&=\mathcal{N}_{\operatorname{cKE}}((\omega_{i}+dd^c\varphi_i),(\omega_{i}))+\operatorname{Ent}\left(\frac{\left(\sum_{i=0}^m\omega_i+dd^c\varphi_i\right)^n}{\int_X\left(\sum_{i=0}^m\omega_i\right)^n},\frac{e^{-\sum_{i=0}^m\varphi_i}\Omega^n}{\int_Xe^{-\sum_{i=0}^m\varphi_i}\Omega^n}\right)-R'
	\end{align*}holds for any $\varphi_i\in \mathcal{H}(\omega_i)$, where $R':=\operatorname{Ent}\left(\frac{\left(\sum_{i=0}^m\omega_i\right)^n}{\int_X\left(\sum_{i=0}^m\omega_i\right)^n},\frac{\Omega^n}{\int_X\Omega^n}\right)$.
\end{theorem}

\begin{proof}
	By Proposition \ref{our5}, we have\begin{align*}
		&\mathcal{M}'((\omega_{i}+dd^c\varphi_i),(\omega_{i}))-\mathcal{N}_{\operatorname{cKE}}((\omega_{i}+dd^c\varphi_i),(\omega_{i}))=\mathcal{M}_{\operatorname{cscK}}\left(\sum_{i=0}^m(\omega_i+dd^c\varphi_i),\sum_{i=0}^m\omega_i\right)\\&-\sum_{i=0}^m\frac{1}{V_i(n+1)!}\braket{\varphi_{i},\cdots,\varphi_{i}}_{(\omega_i,\cdots,\omega_i)}+\frac{1}{V(n+1)!}\bigg<\sum_{i=0}^m\varphi_{i},\cdots,\sum_{i=0}^m\varphi_{i}\bigg>_{\left(\sum_{i=0}^m\omega_i,\cdots,\sum_{i=0}^m\omega_i\right)}\\&+\sum_{i=0}^m\frac{1}{(n+1)!V_i}\sum_{j=0}^n\int_X\varphi_i\omega_i^j\wedge(\omega_i+dd^c\varphi_i)^{n-j}+\frac{1}{n!}\log\left(\int_Xe^{-\sum_{i=0}^m\varphi_i}\Omega^n\right)\\ = & \ \mathcal{M}_{\operatorname{cscK}}\left(\sum_{i=0}^m(\omega_i+dd^c\varphi_i),\sum_{i=0}^m\omega_i\right)\\&+\frac{1}{V(n+1)!}\bigg<\sum_{i=0}^m\varphi_{i},\cdots,\sum_{i=0}^m\varphi_{i}\bigg>_{\left(\sum_{i=0}^m\omega_i,\cdots,\sum_{i=0}^m\omega_i\right)}+\frac{1}{n!}\log\left(\int_Xe^{-\sum_{i=0}^m\varphi_i}\Omega^n\right)\\ =&\  \mathcal{M}_{\operatorname{cscK}}\left(\sum_{i=0}^m(\omega_i+dd^c\varphi_i),\sum_{i=0}^m\omega_i\right)-\mathcal{N}_{\operatorname{KE}}\left(\sum_{i=0}^m(\omega_i+dd^c\varphi_i),\sum_{i=0}^m\omega_i\right)\\=&\ \operatorname{Ent}\left(\frac{\left(\sum_{i=0}^m\omega_i+dd^c\varphi_i\right)^n}{\int_X\left(\sum_{i=0}^m\omega_i\right)^n},\frac{e^{-\sum_{i=0}^m\varphi_i}\Omega^n}{\int_Xe^{-\sum_{i=0}^m\varphi_i}\Omega^n}\right)-R'.
	\end{align*}The final step comes from the relationship between $\mathcal{M}_{\operatorname{cscK}}$ and $\mathcal{N}_{\operatorname{KE}}$ in the classical KE case (see for example \cite[(5.15)]{tian1997kahler} or \cite[Lemma 4.4]{BBEGZ19}).
\end{proof}
\begin{proof}[Proof of Corollary \ref{our4}]
	If $([\omega_i])$ admits a coupled K\"{a}hler-Einstein metric, by Theorem \ref{DingBoundedBelow} (iii) and \ref{ourt2}, we know $\mathcal{N}_{\operatorname{cKE}}(\cdot,(\omega_i))\geq C$ and so are $\mathcal{M}_{\operatorname{ccscK}}(\cdot,(\omega_i))$ and $\mathcal{M}'(\cdot,(\omega_i))$. It means\begin{align*}
		\lim_{t\to\infty}\frac{\mathcal{M}_{\operatorname{ccscK}}\left((\omega_i(t)),(\omega_i)\right)}{t}\geq 0,\ \lim_{t\to\infty}\frac{\mathcal{M}'\left((\omega_i(t)),(\omega_i)\right)}{t}\geq 0
	\end{align*}for any smooth ray $(\omega_i(t))$ in $\prod_{i=0}^{m}\mathcal{H}(\omega_i)$. In particular, if we apply (\ref{his2}) on it, then
	\begin{align*}
		&-\sum_{i=0}^m\frac{1}{V_i(n+1)!}(\mathcal{L}_{i}^{n+1})+\frac{1}{V_0n!}\left(\left(K_{\mathcal{X}_0/\mathbb{P}^1}+\sum_{i=0}^m\mathcal{L}_i\right)\cdot\mathcal{L}_0^n\right)+\frac{1}{V_0n!}((\mathcal{X}_0(0)_{\operatorname{red}}-\mathcal{X}_0(0))\cdot\mathcal{L}_0^n)
	\end{align*}is non-negative, for any smooth ample algebraic test conﬁgurations $(\mathcal{X}_i,\mathcal{L}_i)$ of $(X,L_i)$ that dominates $X\times\mathbb{P}^1$. Then (\ref{his1}) follows since $((\mathcal{X}_0(0)_{\operatorname{red}}-\mathcal{X}_0(0))\cdot\mathcal{L}_0^n)\leq 0$. Similarly using Proposition \ref{our5} (ii), we know that\begin{align*}
		&-\sum_{i=0}^m\frac{1}{V_i(n+1)!}(\mathcal{L}_{i}^{n+1})+\frac{1}{Vn!}\left(\left(K_{\mathcal{X}/\mathbb{P}^1}+\sum_{i=0}^m\mathcal{L}_i\right)\cdot\left(\sum_{i=0}^m\mathcal{L}_i\right)^n\right)\\&+\frac{1}{Vn!}\left((\mathcal{X}(0)_{\operatorname{red}}-\mathcal{X}(0))\cdot\left(\sum_{i=0}^m\mathcal{L}_i\right)^n\right)\geq 0,
	\end{align*}which yields the result (\ref{his2}).
\end{proof}
\subsection{Along holomorphic vector fields}
\begin{theorem}Suppose that $c_1(X)=\sum_{i=0}^m[\omega_i]$. Let $g_t,t\in [0,1]$ be a smooth path in $\operatorname{Aut}(X)$, with $g_0=\operatorname{id}$, and $\omega_{i,t}:=g_t^*\omega_i$, then we have\begin{align}\label{6}
		\mathcal{M}_{\operatorname{ccscK}}((g_t^*\omega_i),(\omega_{i}))=\mathcal{N}_{\operatorname{cKE}}((g_t^*\omega_i),(\omega_{i}))=\mathcal{M}'((g_t^*\omega_i),(\omega_i)).
	\end{align}
\end{theorem}
\begin{proof}
	By ddbar-Lemma, there exists smooth families of $\omega_i$-psh functions $\psi_{i,t},i=0,...,m, t\in [0,1]$ such that\begin{align*}
		\omega_{i,t}=\omega_i+dd^c\psi_{i,t}.
	\end{align*}
	
	Set\begin{align}\label{0}
		\varphi_{i,t}=\psi_{i,t}+\frac{1}{m+1}\operatorname{log}\left(\int_Xe^{-\sum_{i=0}^m\psi_{i,t}}\Omega^n\right),
	\end{align}then\begin{align*}
		\omega_{i,t}=\omega_i+dd^c\varphi_{i,t},\ \ \int_Xe^{-\sum_{i=0}^m\varphi_{i,t}}\Omega^n=1.
	\end{align*}
	Note that \begin{align*}
		-\frac{\sqrt{-1}}{2\pi}\partial\bar{\partial} \log(e^{-\sum_{i=0}^m\varphi_{i,t}}\Omega^n)=\sum_{i=0}^m(\omega_i+dd^c\varphi_{i,t})=g^*_t\left(\sum_{i=0}^m\omega_i\right)=-\frac{\sqrt{-1}}{2\pi}\partial\bar{\partial} \log\left(g_t^*\Omega^n\right)
	\end{align*}and $\int_Xe^{-\sum_{i=0}^m\varphi_{i,t}}\Omega^n=1=\int_Xg_t^*\Omega^n$, thus\begin{align}\label{1}
		e^{-\sum_{i=0}^m\varphi_{i,t}}\Omega^n=g_t^*\Omega^n.
	\end{align}
	
	Now consider the path $(\varphi_{i,t})\subset \prod_{i=0}^m\mathcal{H}(\omega_i)$. Along this path, we have\begin{align*}
		&\operatorname{Ent}\left(\frac{(\omega_0+dd^c\varphi_{0,t})^n}{\int_X\omega_0^n},\frac{e^{-\sum_{i=0}^m\varphi_{i,t}}\Omega^n}{\int_Xe^{-\sum_{i=0}^m\varphi_{i,t}}\Omega^n}\right) \overset{\text{(\ref{1})}}{=} \operatorname{Ent}\left(\frac{g_t^*\omega_0^n}{V_0},\frac{g_t^*\Omega^n}{1}\right)
		\\=&\ \int_X\log\left(\frac{g_t^*\omega_0^n}{V_0g_t^*\Omega^n}\right)\frac{g_t^*\omega_0^n}{V_0}=\int_X\log\left(\frac{\omega_0^n}{V_0\Omega^n}\right)\frac{\omega_0^n}{V_0}=\operatorname{Ent}\left(\frac{\omega_0^n}{\int_X\omega_0^n},\frac{\Omega^n}{\int_X\Omega^n}\right).
	\end{align*}Thus by Theorem \ref{ourt2},\begin{align*}
		&\mathcal{M}_{\operatorname{ccscK}}((\omega_{i}+dd^c\varphi_{i,t}),(\omega_{i}))=\mathcal{N}_{\operatorname{cKE}}((\omega_{i}+dd^c\varphi_{i,t}),(\omega_{i}))\\&+\frac{1}{n!}\operatorname{Ent}\left(\frac{(\omega_0+dd^c\varphi_{0,t})^n}{\int_X\omega_0^n},\frac{e^{-\sum_{i=0}^m\varphi_{i,t}}\Omega^n}{\int_Xe^{-\sum_{i=0}^m\varphi_{i,t}}\Omega^n}\right)-\frac{1}{n!}\operatorname{Ent}\left(\frac{\omega_0^n}{\int_X\omega_0^n},\frac{\Omega^n}{\int_X\Omega^n}\right)\\&=\mathcal{N}_{\operatorname{cKE}}((\omega_{i}+dd^c\varphi_{i,t}),(\omega_{i})).
	\end{align*}
	
	Proof of $\mathcal{N}_{\operatorname{cKE}}((g_t^*\omega_i),(\omega_{i}))=\mathcal{M}'((g_t^*\omega_i),(\omega_i))$ is same, since\begin{align*}
		\operatorname{Ent}\left(\frac{\left(\sum_{i=0}^m\omega_i+dd^c\varphi_{i,t}\right)^n}{\int_X\left(\sum_{i=0}^m\omega_i\right)^n},\frac{e^{-\sum_{i=0}^m\varphi_{i,t}}\Omega^n}{\int_Xe^{-\sum_{i=0}^m\varphi_{i,t}}\Omega^n}\right)=\operatorname{Ent}\left(\frac{g_t^*\left(\sum_{i=0}^m\omega_i\right)^n}{\int_X\left(\sum_{i=0}^m\omega_i\right)^n},\frac{g_t^*\Omega^n}{\int_Xe^{-\sum_{i=0}^m\varphi_{i,t}}\Omega^n}\right)
	\end{align*}is constant, for $\varphi_{i,t}$ as above.
\end{proof}

In particular, take $g_t=\exp(t\xi)$ for real holomorphic vector field $\xi$, then
\begin{align}\label{comparison}
	\frac{d}{dt}\bigg|_{t=0}\mathcal{M}_{\operatorname{ccscK}}((\omega_{i}+dd^c\varphi_{i,t}),(\omega_{i}))=\frac{d}{dt}\bigg|_{t=0}\mathcal{N}_{\operatorname{cKE}}((\omega_{i}+dd^c\varphi_{i,t}),(\omega_{i})). 
\end{align}
{By the definition \ref{coupledMabuchi} of $\mathcal{M}_{\operatorname{ccscK}}$} the left hand side of (\ref{comparison}) is equal to
\begin{align*}
	&-\sum_{i=1}^m\int_X\dot{\varphi_{i}}(0)\left(\frac{1}{V_i} \frac{\omega_i^n}{n!}- \frac{1}{V_0}\frac{\omega_0^n}{n!} \right) \\& - \frac{1}{V_0}\int_M \dot{\varphi_0}(0) \left((S_{\omega_0,J}-\hat{S})\frac{\omega_0^n}{n!} - \sum_{i=0}^m\omega_i(t) \wedge \frac{\omega_0^{n-1}}{(n-1)!} \right)	.
\end{align*}Let us compute $\dot{\varphi_i}(0)$, which satisfies\begin{align}\label{2}
	dd^c\dot{\varphi_i}(0)=-\mathcal{L}_{\xi}\omega_i\overset{\text{Cartan}}{=}-d(\iota_{\xi}\omega_i).
\end{align}Set $\xi^{1,0}=\frac{\xi-\sqrt{-1}J\xi}{2}$ and $\xi^{0,1}=\frac{\xi+\sqrt{-1}J\xi}{2}$, then  $\xi^{1,0}$ is a holomorphic $(1,0)$-vector field. Since $X$ is Fano, $b_1(X)=0$ and $\mbox{}^\exists u_i$ such that\begin{align*}
	\iota_{\xi^{1,0}}\omega_i=\sqrt{-1}\bar{\partial}u_i. 
\end{align*}Thus by (\ref{2}), we have\begin{align*}
	dd^c\dot{\varphi_i}(0)=-d(\iota_{\xi}\omega_i)=-\sqrt{-1}d(\bar{\partial}u_i-\partial\bar{u}_i)=-2\pi dd^c(u_i+\bar{u}_i),
\end{align*}
which means $\dot{\varphi_i}(0)=-2\pi dd^c(u_i+\bar{u}_i)+\operatorname{constant}$ and \begin{align*}
	&\frac{d}{dt}\bigg|_{t=0}\mathcal{M}_{\operatorname{ccscK}}((\omega_{i}+dd^c\varphi_{i,t}),(\omega_{i}))\\ =&\ -2\pi\left(\operatorname{Fut}_{c}(\xi^{1,0})+\overline{\operatorname{Fut}_{c}(\xi^{1,0})}\right).
\end{align*}Here $\operatorname{Fut}_{c}(\xi^{1,0})$ is defined by \cite[Definition 10]{datar2020coupled}.
Especially if for all $\omega_i$, $v$ is a holomorphic Killing vector field. Let $\xi=Jv$ and $\phi_i$ be real function such that $\iota_v\omega_i=d\phi_i$, then $\iota_{\xi^{1,0}}\omega_i=\iota_{(Jv)^{1,0}}\omega_i=\sqrt{-1}(\iota_{v}\omega_i)^{0,1}=\sqrt{-1}\bar{\partial}\phi_i$ and\begin{align*}
	\dot{\varphi_i}(0)=-4\pi\phi_i+\operatorname{constant},
\end{align*}where $\varphi_i(t)$ corresponds to $g_t=\exp(t\xi)$. This means
\begin{proposition}
	\begin{align*}
		&\frac{d}{dt}\bigg|_{t=0}\mathcal{M}_{\operatorname{ccscK}}((\omega_{i}+dd^c\varphi_{i,t}),(\omega_{i}))=-4\pi\operatorname{Fut}_c(\xi^{1,0}),
	\end{align*}where $\operatorname{Fut}_{c}(\xi^{1,0})$ is defined by \cite[Definition 10]{datar2020coupled}.
\end{proposition}
On the other hand in \cite{FutakiZhang2021SCM} a coupled Futaki invariant on Fano manifolds is expressed in the following way.
As before, we assume a decomposition $2\pi c_1(X) = \gamma_0 + \cdots + \gamma_m$ and choose K\"ahler forms
$\omega_i \in \gamma_i$. Since both $\operatorname{Ric}(\omega_i)$ and $\sum_{j=0}^m \omega_j$ represent
$c_1(M)$ there exist real smooth functions $f_i$'s such that 
\begin{equation*}
	\operatorname{Ric}(\omega_i) - i\partial\bar{\partial} f_i =  \sum_{j=0}^m \omega_j, \quad i = 0,\ \cdots,\ m.
\end{equation*}
Since $\partial\bar{\partial}\log e^{f_i}\omega_i^n$ are equal for all $i$, by adding suitable constants to $f_i$'s
we may normalize $f_i$ so that 
\begin{equation}\label{normalization}
	e^{f_0}\omega_0^n= \cdots = e^{f_m}\omega_m^n.
\end{equation}
Notice that this is exactly the volume form $\Omega^n$ 
where $\Omega$ be the K\"{a}hler form in $c_1(X)$ such that $\Omega:\operatorname{Ric}(\Omega)=\sum_{i=0}^m\omega_i$:
\begin{equation}\label{normalization2} \Omega = e^{f_0}\omega_0^n= \cdots = e^{f_m}\omega_m^n.  \end{equation}
Let $\xi$ be a holomorphic vector field on $X$. 
Since a Fano manifold is simply connected there exist complex-valued smooth functions defined up to constant
$u_i$ such that 
\begin{equation*}
	i_\xi\omega_i=\bar{\partial}(\sqrt{-1}u_i).
\end{equation*}
It is shown in \cite[Theorem 1.2]{FutakiZhang2021SCM} that with the normalization 
\eqref{normalization} of $f_i$, the Lie algebra $\mathfrak{h}(X)$ of all holomorphic vector fields is isomorphic to the set of all $(m+1)$-tuples of complex valued smooth functions $(u_0,\cdots,u_m)$ satisfying 
\begin{equation}\label{Th3.3}
	\begin{aligned}
		\mathrm{grad}^\prime_i u_i=\mathrm{grad}^\prime_j u_j,\ \ i, j=0, 1,  \dots, m,\\
		\Delta_{i} u_i+(\mathrm{grad}_{i} u_i)f_i = -\sum\limits_{j=0}^m u_j.
	\end{aligned}
\end{equation}
for $i=0, 1,  \dots, m$ where $\Delta_{i}=-\bar\partial^*_{i}\bar\partial$ with respect $\omega_i$.
Then the coupled Futaki invariant is defined by
\begin{align*}
	\mathrm{Fut}: \mathfrak{h}(X)&\to \mathbb C\nonumber\\
	\xi&\mapsto \sum\limits_{i=0}^m\frac{\int_X u_i\ \omega_i^n}{\int_X \omega_i^n} \label{Fut}
\end{align*}
where $\xi = \mathrm{grad}^\prime_i u_i$ for every $i =0, \dots, m$.
It was also proved that 
\begin{enumerate}
	\item[(1)]$\mathrm{Fut}(\xi)$ is independent of the choice of K\"ahler forms $\omega_i\in \gamma_i$, $i=1, \dots, k$.
	\item[(2)] If $X$ admits coupled K\"ahler-Einstein metrics then $\mathrm{Fut}= 0$.
\end{enumerate}

Note that, from (\ref{normalization2}) and (\ref{Th3.3}), we have
\begin{equation}\label{normalization3}
	\int_X (u_0 + \cdots + u_m)\,\Omega^n = 0.
\end{equation}
Using the computation above, (\ref{normalization3}) and (\ref{Ding}) one can show for a holomorphic Killing vector field $v$ with $\xi = Jv$
$$ \frac{d}{dt}\bigg|_{t=0}\mathcal{N}_{\operatorname{cKE}}((\omega_{i}+dd^c\varphi_{i,t}),(\omega_{i}))
= - \frac{1}{n+1} \sum_{i=0}^m \frac{\int_X u_i\ \omega_i^n}{V_i} = - \frac{1}{n+1}\operatorname{Fut} (\xi).
$$
Hence by (\ref{comparison}) we obtain the following:
\begin{proposition}\label{coincide}
	On a Fano manifold, the coupled Futaki invariants in \cite[(1.9)]{FutakiZhang2021SCM} and \cite[Definition 10]{datar2020coupled} coincide up to a constant.
\end{proposition}
\begin{remark}
	Proposition \ref{coincide} can also be derived from \cite[Lemma 18]{datar2020coupled}, where they claim\begin{align*}
		\operatorname{Fut}_c(X,\xi)=\sum_{i=0}^{m}\frac{\int_X u'_i\ \omega_i^n}{\int_X \omega_i^n}-\frac{\int_X\sum_{i=0}^{m}u_i'\ \Omega^n}{\int_X\Omega^n}.
	\end{align*}Here $u_i'$ satisfies $\iota_{\xi}\omega_i=\bar{\partial}(\sqrt{-1}u_i')$. When $u_i'$ is chosen as (\ref{Th3.3}), then by (\ref{normalization3})  $\int_{X}\sum_{j=0}^mu_j'\Omega=0$ and $\operatorname{Fut}_c(X,\xi)=\sum_{i=0}^{m}\frac{\int_X u'_i\ \omega_i^n}{\int_X \omega_i^n}$ equals $\operatorname{Fut}(\xi)$ up to a constant.
\end{remark}
\color{black}
\subsection{Along product test configurations}
Now suppose that $[\omega_i]=L_i$ for ample line bundle $L_i$, $\sum_{i=0}^mL_i=-K_X$, and $(X,(L_i))$ admits a $\mathbb{C}^*$-action. This gives a tuple of product (not compactified) test configuration $(\mathring{\mathcal{X}},(\mathring{\mathcal{L}}_i),\mathring{\mathcal{L}})=(X\times \mathbb{C}^*,\operatorname{pr}^*_1L_0,\cdots,\operatorname{pr}^*_1L_m,\otimes_{i=0}^m\operatorname{pr}_1^*L_i)$. Using the $\mathbb{C}^*$-action we can compactify them by $(\mathcal{X},(\mathcal{L}_i),\mathcal{L})$.

Let $\xi$ be the real holomorphic vector field that generates the $\mathbb{C}^*$-action on $X$ and $g_t=\exp(-t\xi)$. Take $\mathbb{S}^1$-invariant positive metric $h_i$ on $L_i$, such that $\omega_i=dd^ch_i$ is $\mathbb{S}^1$-invariant K\"{a}hler form on $X$. Then $v=J\xi$ is holomorphic Killing vector field for all $\omega_i$.	
\begin{proposition}Let $\omega_{i,t}=g_t^*\omega_i$, then
	\begin{align*}
		\lim_{t\to\infty}\frac{\mathcal{M}_{\operatorname{ccscK}}\left((\omega_{i,t}),(\omega_i)\right)}{t}&=-\sum_{i=0}^m\frac{1}{V_i(n+1)!}(\mathcal{L}_i^{n+1})+\frac{1}{V_0n!}\left(\left(K_{\mathcal{X}/\mathbb{P}^1}+\sum_{i=0}^m\mathcal{L}_i\right)\cdot\mathcal{L}_0^n\right),
	\end{align*}and\begin{align*}
		\lim_{t\to\infty}\frac{\mathcal{M}'\left((\omega_{i,t}),(\omega_i)\right)}{t}&=-\sum_{i=0}^m\frac{1}{V_i(n+1)!}(\mathcal{L}_i^{n+1})+\frac{1}{Vn!}\left(\left(K_{\mathcal{X}/\mathbb{P}^1}+\sum_{i=0}^m\mathcal{L}_i\right)\cdot\left(\sum_{i=0}^m\mathcal{L}_i\right)^n\right).
	\end{align*}
\end{proposition}
\begin{proof}
	We denote the inverse of gluing map from $(X\times(\mathbb{P}^1\backslash\{0\}),\operatorname{pr}_1^*L)$ to $(\mathcal{X}|_{\mathbb{P}^1\backslash\{0\}},\mathcal{L}|_{\mathbb{P}^1\backslash\{0\}})$ by $\mu$. Suppose that the $\mathbb{C}^*$-action on $X$ is $\lambda$, then $\mu:X\times\mathbb{C}^*=\mathring{\mathcal{X}}\backslash\mathring{\mathcal{X}}_0\simeq X\times \mathbb{C}^*$ is given by\begin{align*}
		(x,a)\mapsto (\lambda^{-1}(a).x,a).
	\end{align*}To distinguish between these two charts, we denote the projection from $\mathring{\mathcal{X}}$ to $X$ by $\pi_1$, and the projection from $\mathcal{X}|_{\mathbb{P}^1\backslash\{0\}}$ to $X$ by $\operatorname{pr}_1$, namely $$
	\begin{tikzcd}
		\mathring{\mathcal{L}}_i=\pi_1^*L_i \arrow[r] & \mathring{\mathcal{X}}=X\times\mathbb{C} \arrow[d, "\pi_1"'] & X\times\mathbb{C}^* \arrow[l, hook] \arrow[r, "\mu", equal] & X\times\mathbb{C}^* \arrow[r, hook] & X\times(\mathbb{P}^1\backslash\{0\}) \arrow[d, "\operatorname{pr}_1"] \\
		L_i \arrow[r]                    & X                                                 &                                                                           &                                     & X                                                                    
	\end{tikzcd}
	$$
	
	Recall that $g_t=\exp(-t\xi)$ where $\xi$ generates $\lambda$, this means\begin{align*}
		\mu|_{\mathring{\mathcal{X}}_{\exp(-t)}}=\lambda^{-1}(e^{-t})=\exp(t\xi)=g_t^{-1}.
	\end{align*} Note that the Hermitian metric $h_i$ can be defined on $\mathring{\mathcal{L}}_i$ (not $\mathcal{L}$) by $\pi_1^*h_i$, which gives an $\mathbb{S}^1$-invariant form $dd^ch_i=\pi_1^*\omega_i$ on $\mathring{\mathcal{X}}$. Since $\mu$ is equivariant, $\mu_*(dd^ch_i)$ is $\mathbb{S}^1$-invariant under the trivial action of $X\times \mathbb{C}^*$. At this time we have\begin{align*}
		(\mu_*(dd^ch_i))|_{X\times \{e^{-t}\}}=(\mu|_{\mathring{\mathcal{X}}_{\exp(-t)}})_*(\omega_i) = g_t^*\omega_i.
	\end{align*}We claim that
	\begin{lemma}
		Consider $\omega_{i,t}=g_t^*\omega_i$, then there exists $\psi_{i,t}$ such that $\omega_{i,t}=\omega_i+dd^c\psi_{i,t}$ and $\psi_{i,t}$ is a geodesic $\mathcal{C}^\infty$-compatible with $(\mathcal{X},\mathcal{L}_i)$.
	\end{lemma}
	\begin{proof}
		At the expense of some information outside $\mathring{\mathcal{X}}|_{\Delta}$, we can extend $h_i$ on $\mathring{\mathcal{L}}_i$ to an $\mathbb{S}^1$-invariant Hermitian metric $H_i$ on $\mathcal{L}_i$. Set $\Phi_i=dd^cH_i\in c_1(\mathcal{L}_i)$. Now $\mu^*\operatorname{pr}_1^*h_i$ and $H_i$ are two Hermitian metrics on $\mathcal{L}_i|_{\mathbb{P}^1\backslash\{0\}}$, so there exists $\Psi_i:(X\times \mathbb{P}^1\backslash\{0\})\to \mathbb{R}$ such that\begin{align*}
			\Phi_i=\mu^*\operatorname{pr}_1^* dd^ch_i+\mu^*dd^c\Psi_i=\mu^*\operatorname{pr}_1^*\omega_i+\mu^*dd^c\Psi_i
		\end{align*}on $\mathcal{X}|_{\mathbb{P}^1\backslash\{0\}}$. Take $\psi_{i,t}:=\Psi_i(\cdot, e^{-t})$. When restrict to $\mathring{\mathcal{X}}|_{\exp(-t)}$, we have\begin{align*}
			(\Phi_i)|_{\mathring{\mathcal{X}}_{\exp(-t)}}&=(\mu^*\operatorname{pr}_1^*\omega_i)|_{\mathring{\mathcal{X}}_{\exp(-t)}}+(\mu|_{\mathring{\mathcal{X}}_{\exp(-t)}})^*dd^c\psi_{i,t}\\&=(\mu|_{\mathring{\mathcal{X}}_{\exp(-t)}})^*\omega_i+(\mu|_{\mathring{\mathcal{X}}_{\exp(-t)}})^*dd^c\psi_{i,t}.
		\end{align*}Thus\begin{align*}
			\omega_i+dd^c\psi_{i,t}=(\mu|_{\mathring{\mathcal{X}}_{\exp(-t)}})_*(\Phi_i|_{\mathring{\mathcal{X}}_{\exp(-t)}})=g_t^*\omega_i.
		\end{align*}Finally, since $\Phi_i=\pi_1^*\omega_i$, $\Phi_i^{n+1}=0$ and thus $(\operatorname{pr}_1^*\omega_i+dd^c\Psi_i)^{n+1}=0$. All above shows that $\psi_{i,t}$ is a smooth geodesic $\mathcal{C}^\infty$-compatible with $(\mathcal{X},\mathcal{L}_i)$.
	\end{proof}The proof is thus completed by invoking Lemma \ref{slopeofAnalyticDelignePairing} and \cite[Theorem A]{boucksom2019}. Note that $\mathcal{X}$ is given by product, which means $\mathring{\mathcal{X}}_{0,\operatorname{red}}=\mathring{\mathcal{X}}_0$.
\end{proof}
\begin{corollary}\label{strange1}
	We have\begin{align*}
		4\pi\operatorname{Fut}_c(\xi^{1,0})=&-\sum_{i=0}^m\frac{1}{V_i(n+1)!}(\mathcal{L}_i^{n+1})+\frac{1}{V_0n!}\left(\left(K_{\mathcal{X}/\mathbb{P}^1}+\sum_{i=0}^m\mathcal{L}_i\right)\cdot\mathcal{L}_0^n\right)\\=&-\sum_{i=0}^m\frac{1}{V_i(n+1)!}(\mathcal{L}_i^{n+1})+\frac{1}{Vn!}\left(\left(K_{\mathcal{X}/\mathbb{P}^1}+\sum_{i=0}^m\mathcal{L}_i\right)\cdot\left(\sum_{i=0}^m\mathcal{L}_i\right)^n\right)
	\end{align*}for the tuple of test configurations $(\mathcal{X},(\mathcal{L}_i))$ compactified from the product test configuration as above.
\end{corollary}
\begin{proof}
	Note that for $\varphi_{i,t}$ defined by ($\ref{0}$), \begin{align*}
		dd^c\dot{\varphi_{i}}(a)=-\mathcal{L}_{\xi}(g_a^*\omega)=g_a^*(-\mathcal{L}_{\xi}\omega)=g_a^*dd^c\dot{\varphi_i}(0).
	\end{align*}Thus\begin{align*}
		&\frac{d}{dt}\bigg|_{t=a}\mathcal{M}_{\operatorname{ccscK}}((\omega_{i}+dd^c\varphi_{i,t}),(\omega_{i}))
		\\=&\ -\sum_{i=1}^m\int_X\dot{\varphi_{i}}(a)\left(\frac{1}{V_i} \frac{g_a^*\omega_i^n}{n!}- \frac{1}{V_0}\frac{g_a^*\omega_0^n}{n!} \right) 
		\\& - \frac{1}{V_0}\int_M \dot{\varphi_0}(a) \left((S_{g_a^*\omega_0,J}-\hat{S})\frac{g_a^*\omega_0^n}{n!} - \sum_{i=0}^mg_t^*\omega_i(t) \wedge \frac{g_a^*\omega_0^{n-1}}{(n-1)!} \right)
		\\=&\ -\sum_{i=1}^m\int_X\dot{\varphi_{i}}(0)\left(\frac{1}{V_i} \frac{\omega_i^n}{n!}- \frac{1}{V_0}\frac{\omega_0^n}{n!} \right) 
		\\& - \frac{1}{V_0}\int_M \dot{\varphi_0}(0) \left((S_{\omega_0,J}-\hat{S})\frac{\omega_0^n}{n!} - \sum_{i=0}^m\omega_i(t) \wedge \frac{\omega_0^{n-1}}{(n-1)!} \right)	
		\\=&\ \frac{d}{dt}\bigg|_{t=0}\mathcal{M}_{\operatorname{ccscK}}((\omega_{i}+dd^c\varphi_{i,t}),(\omega_{i})).
	\end{align*}This means $\mathcal{M}_{\operatorname{ccscK}}$ is affine along $\omega_{i,t}$.
\end{proof}
\begin{remark}
	By Corollary \ref{strange1} we have\begin{align*}
		\frac{1}{V_0n!}\left(\left(K_{\mathcal{X}/\mathbb{P}^1}+\sum_{i=0}^m\mathcal{L}_i\right)\cdot\mathcal{L}_0^n\right)=\frac{1}{Vn!}\left(\left(K_{\mathcal{X}/\mathbb{P}^1}+\sum_{i=0}^m\mathcal{L}_i\right)\cdot\left(\sum_{i=0}^m\mathcal{L}_i\right)^n\right)
	\end{align*}for the compactified product test configuration. In fact both of them equal to a constant $c/n!$ that depends on $\left(\mathcal{X},\sum_{i=0}^m\mathcal{L}_i\right)$, where $c$ is given by\begin{align*}
		K_{\mathcal{X}/\mathbb{P}^1}+\sum_{i=0}^m\mathcal{L}_i=\pi^*\mathcal{O}_{\mathbb{P}^1}(c),
	\end{align*}\begin{proof}
		Let $\mathring{\mathcal{L}}:=\sum_{i=0}^m\mathring{\mathcal{L}}_i$, then $(\mathring{\mathcal{X}},\mathring{\mathcal{L}})$ is a product test configuration of $(X,-K_X)$. We denote the $\mathbb{C}^*$-action on $X,-K_X$ that induces this product test configuration by $\lambda,\mu$. Note that the $\mathbb{C}^*$-action $\lambda$ induces another action $\Lambda$ on $-K_X$, which is actually the action of $(\mathcal{X},-K_{\overline{X}/\mathbb{P}^1})$.
		
		Now restrict to $\mathring{\mathcal{X}}$, both $\mathring{\mathcal{L}}$ and $-K_{\mathring{\mathcal{X}}/\mathbb{C}}$ are isomorphic to $-\operatorname{pr}_1^*K_X$, but with different actions $\mu,\Lambda$. Since $X$ is compact, $\Lambda^{-1}(a)\circ\mu(a)$ is just constant scaling. Group axiom then implies that $\Lambda(a)=a^c\mu(a)$ for some integer $c$. This means that the gluing (transform) function of $\mathring{\mathcal{L}}$ differs by a factor of $*^c$ from that of $-K_{\mathcal{X}/\mathbb{P}^1}$.	
	\end{proof}
\end{remark}

\bibliographystyle{spmpsci.bst}
\bibliography{references}
\end{document}